\documentclass[11pt,leqno]{amsart}

\usepackage[utf8x]{inputenc}
\usepackage[T1]{fontenc}
\usepackage[english]{babel}

\usepackage[margin=1.1in]{geometry}

\usepackage{amsmath,amssymb,amsthm}
\usepackage{mathabx}
\usepackage{yfonts}
\usepackage{mathrsfs}
\usepackage{faktor}
\usepackage[lite,initials, alphabetic]{amsrefs}
\usepackage[dvipsnames]{xcolor}
\usepackage[colorlinks=true, linkcolor=MidnightBlue, citecolor=OliveGreen]{hyperref}
\usepackage{comment}

\usepackage{graphicx}
\usepackage{tikz}
\usepackage{tikz-cd}
\usetikzlibrary{arrows, automata, positioning}
\usepackage{nicefrac}

\theoremstyle{definition}
\newtheorem{defn}{{\bf Definition}}[section]

\newtheorem{rmk}[defn]{{\bf Remark}}
\theoremstyle{theorem}
\newtheorem{lemma}[defn]{{\bf Lemma}}
\newtheorem{theorem}[defn]{{\bf Theorem}}

\newtheorem{prop}[defn]{{\bf Proposition}}

\DeclareMathOperator{\N}{\mathbb{N}_0}

\DeclareMathOperator{\Sym}{\mathrm{Sym}}
\DeclareMathOperator{\Aut}{Aut}

\DeclareMathOperator{\st}{st}
\DeclareMathOperator{\St}{St}
\DeclareMathOperator{\rst}{rist}

\newcommand{\LD}{\langle}
\newcommand{\RD}{\rangle}

\newcommand{\intseg}[2]{{[#1, #2]}} 

\numberwithin{equation}{section}

\DeclareMathOperator{\bp}{Bas}
\newcommand{\cref}[2]{\hyperref[#2]{#1~\ref*{#2}}}

\title[Maximal subgroups of iterated monodromy groups]{Maximal subgroups of a family of \\
iterated monodromy groups 
}

\author[K. Rajeev]{Karthika Rajeev} 
\address{Karthika Rajeev: Fakult\"at f\"ur Mathematik,
Universit\"at Bielefeld,
33501 Bielefeld, Germany}
\email{krajeev@math.uni-bielefeld.de}

\author[A. Thillaisundaram]{Anitha Thillaisundaram} 
\address{Anitha Thillaisundaram: Centre for Mathematical Sciences, Lund University,  
223 62 Lund, Sweden}
\email{anitha.thillaisundaram@math.lu.se}

\date{\today}

\thanks{This research was partially conducted in the framework of the Deutsche 
Forschungsgemeinschaft (DFG, German Research Foundation)-funded research 
training group GRK 2240: Algebro-Geometric Methods in Algebra, 
Arithmetic and Topology and by the DFG – Project-ID 491392403 – TRR 358. The second author thanks the Heinrich-Heine-University of D\"{u}sseldorf and Bielefeld University for their hospitality.}

\keywords{Groups acting on rooted trees, iterated monodromy groups, weakly branch groups, maximal subgroups}

\subjclass[2010]{Primary  20E08;  Secondary 20E28}

\begin{document}

\maketitle

\begin{abstract}
The Basilica group is a well-known 2-generated weakly branch, but not branch, group acting on the binary rooted tree. Recently a more general form of the Basilica group has been investigated by Petschick and Rajeev, which is an $s$-generated weakly branch, but not branch, group that acts on the $m$-adic tree, for $s,m\ge 2$. A  larger family of groups, which contains these generalised Basilica groups, is the family of iterated monodromy groups. With the new developments by Francoeur, the study of the existence of maximal subgroups of infinite index has been extended from branch groups to weakly branch groups. Here we show that a subfamily of iterated monodromy groups, which more closely resemble the generalised Basilica groups, have maximal subgroups only of finite index.
\end{abstract}

\section{Introduction}

Groups acting on rooted trees have drawn a great deal of attention over the last couple of decades because they exhibit prominent features and solve several long-standing problems in group theory. The initial examples studied were Grigorchuk's groups of intermediate word growth (\cite{Gri80}; answering Milnor's question) and Gupta and Sidki's examples of finitely generated infinite $p$-groups (\cite{GS83}; providing an explicit family of $2$-generated counterexamples to the general Burnside problem). Ever since, attempts have been made to characterise and generalise the groups of automorphisms of rooted trees. Today, the Grigorchuk groups and the Gupta--Sidki groups are known as the first examples of groups in the family of \emph{branch groups}. Branch groups are groups acting level-transitively on a spherically homogeneous rooted tree~$T$ and having subnormal subgroups similar to that of the full automorphism group $\Aut T$ of the tree~$T$; see \cref{Section}{sec:pre} for definitions. The groups studied in this paper belong to a more general class of groups, the \emph{weakly branch groups}, obtained by weakening some of the algebraic properties of branch groups; cf.~\cite{Handbook}.

The Basilica group is a $2$-generated weakly branch, but not branch, group acting on the binary rooted tree, which was introduced by Grigorchuk and \.{Z}uk in \cite{GZ01} and \cite{GZ02}. 
It is the first known example of an amenable \cite{BV05} but not sub-exponentially amenable group \cite{GZ02}. In contrast to the Grigorchuk and the Gupta--Sidki groups, the Basilica group is torsion-free and has exponential word growth \cite{GZ02}. Moreover, it is the iterated monodromy group of the complex polynomial $z^2-1$; see~\cite[Section~6.12.1]{Nekrashevych}.
The generators of the Basilica group are recursively defined as follows:
  \begin{align*}
      	a = (1 , b) \quad\text{ and }\quad b = (1, a) \sigma,
  \end{align*}
  where $\sigma$ is the cyclic permutation which swaps the subtrees rooted at the first level of the binary rooted tree, and $(x,y)$ represents the independent action on the two maximal subtrees, where $x,y\in \Aut T$. 
  Recently, Petschick and Rajeev~\cite{PR} introduced a construction which relates the Basilica group and the one-generated dyadic odometer $\mathcal{O}_2$ (also known as the adding machine). Let $m, \,s \geq 2$ be integers and let $G$ be a subgroup of the automorphism group $\Aut T$ of the $m$-adic tree~$T$. The $s$th Basilica group of $G$ is given by $\bp_s(G) = \LD\beta^s_i(g)\mid g \in G, \, i \in \{0,1,\dots, s-1\}\,\RD$, where $\beta_i^s: \Aut T \rightarrow \Aut T$ are monomorphisms given by
\begin{align*}
    \beta^s_i(g) &= (1, \dots, 1,\beta^s_{i-1}(g)) &\text{ for } i \in \{1,\dots,s-1\},\\
		\beta^s_0(g) &=  (\beta^s_{s-1}(g_0), \dots, \beta^s_{s-1}(g_{m-1}))g^\epsilon,
\end{align*}
 where $g_x$ is the restriction of~$g$ to the subtree rooted at a first-level vertex $x\in \{0,\ldots,m-1\}$, and  $g^\epsilon$ is the local action of the element~$g$ at the root of~$T$. (In~\cite{PR} the generators $\beta_i^s(g)$, for $i \in \{1,\dots,s-1\}$, are defined along the leftmost spine and the element $g^\epsilon$ is acting from the left, however, the corresponding $s$th Basilica group is isomorphic to the one defined here).
 We obtain the classical Basilica group by applying the operator $\bp_2$ to the dyadic odometer as follows:  let $c = (1,c) \sigma$ be the automorphism of the binary rooted tree generating the dyadic odometer. Then the generators of the Basilica group are given by
  \begin{align*}
      a = \beta_1^2(c)\quad \text{ and }\quad b = \beta_0^2(c).
  \end{align*}
 This gives a natural generalisation of the Basilica group given by $\bp_s(\mathcal{O}_m)$ for every pair of integers $m\, , s \geq 2$. Here $\mathcal{O}_m$ is the $m$-adic odometer, which is an embedding of the infinite cyclic group into the automorphism group of the $m$-adic tree~$T$, and is generated by 
\[
   c = (1,\overset{m-1}\dots,1,c)\sigma
 \]
 where $\sigma = (0 \; 1 \; \cdots \; m-1)$ is the $m$-cycle that cyclically permutes the $m$ subtrees rooted at the first level of~$T$. The generalised Basilica groups $\bp_s(\mathcal{O}_m)$ resemble the classical Basilica group, as they are weakly branch, but not branch, torsion-free groups of exponential word growth \cite[Theorem~1.6]{PR}. They are also weakly regular branch over their derived subgroup. 

Now, as we will see below, the generalised Basilica groups lie in the set of iterated monodromy groups of post-critically finite complex polynomials, where a polynomial $f$ is \emph{post-critically finite} if the orbit of the critical point 0 under iterations of $f$ is finite. Iterated monodromy groups are more naturally defined in terms of self-coverings of topological spaces, as done in~\cite[Chapter~5]{Nekrashevych}, and these groups are of interest since they encode information about the dynamics of such self-coverings. However below we will only abstractly define a subfamily of iterated monodromy groups that are of interest to us, and we only study their algebraic properties. Therefore we refer the reader to \cite{Nekrashevych} for more information about iterated monodromy groups in general.

 We note  that the iterated monodromy groups of the quadratic polynomials $f(z) = z^2+c$ have been studied in more depth. In \cite{BN08}, Bartholdi and Nekrashevych identified the iterated monodromy groups of post-critically finite quadratic polynomials $f(z) = z^2+c$. 
 Among those groups we are interested in the groups $\mathfrak{K(v)}$, which are defined as below. 

 Let $s \geq 2$ and ${\mathfrak{v}} = x_0\cdots x_{s-2}$ be a word over the alphabet $X = \{0,1\}.$ The group $\mathfrak{K(v)}$ is a subgroup of automorphisms of the binary rooted tree generated by the elements $a_0,\dots ,a_{s-1}$,  which are defined as 
 \begin{align*}
    a_0 = (1,a_{s-1}) \sigma, && a_{i+1} = \begin{cases}
                                   (a_{i},1) &\text{if } x_i = 0, \\
                                   (1, a_{i}) &\text{if } x_i =1,
                                \end{cases}
 \end{align*}
 for $i \in {\{0,1,\ldots, s-2\}}$.
 The quadratic polynomial associated to $\mathfrak{K(v)}$ can be precisely defined in terms of~${\mathfrak{v}}$; see 
 \cite[Section 5]{BN08}. Note that if $s=1$, 
 then ${\mathfrak{v}}$ is the  empty word, then $\mathfrak{K}(\varnothing)$ is just the cyclic group  $\mathbb{Z}$.
 Notice that $\mathfrak{K}(1)$ is the classical Basilica group. Furthermore, by setting ${\mathfrak{v}} = 1 \overset{s'}\dots 1$, we get that $\mathfrak{K(v)} = \bp_{s'+1}(\mathcal{O}_2)$, for ${s'}\in \mathbb{N}$.  The group $\mathfrak{K(v)}$ has a corresponding so-called kneading sequence, by which the symbolic dynamics of  quadratic complex polynomials are usually studied; see~\cite{BN08} for more information.
 
 The recurrence relations in the definition of $\mathfrak{K(v)}$ above suggests that one can generalise this notion to groups acting on the $m$-adic tree, for any $m\ge 2$. Let ${\mathfrak{v}} = x_0 \cdots x_{s-2}$ be a word in the alphabet $X = \{0,1,\dots,m-1\}$. Define
\begin{align*}
     a_0 = (1,\dots,1,a_{s-1}) \sigma, {\qquad\text{and}\qquad a_{i+1} = (1,\overset{x_i-1}{\ldots},1,                                   a_{i}, 1,\ldots,1),}
\end{align*}
for $i \in {\{0,1,\ldots, s-2\}}$, and $\mathfrak{K(v)}$ as the group defined by the elements $a_0,\dots,a_{s-1}$. Then, for ${\mathfrak{v}} = 1 \overset{s'}\dots 1$, we have $\mathfrak{K(v)} = \bp_{s'+1}(\mathcal{O}_m)$. 

 We establish the following basic properties of the groups $\mathfrak{K(v)}$.

 \begin{theorem}\label{thm:properties}
  Let $m,s\ge 2$  
  be positive integers, and let ${\mathfrak{v}} = x_0\cdots x_{s-2}$ be a word in the alphabet $X=\{0,1 ,\ldots, m-1\}$. Then for the group $\mathfrak{K(v)}$ defined by ${\mathfrak{v}}$, the following assertions hold: 
\begin{itemize}
\item[(i)] $\mathfrak{K(v)}$ is an iterated monodromy group of a post-critically finite polynomial; 
    \item[(ii)] $\mathfrak{K(v)}$ is super strongly fractal and level-transitive;
    \item[(iii)] $\mathfrak{K(v)}$ is not branch, but is weakly regular branch over its commutator subgroup $\mathfrak{K(v)}'$;
    \item[(iv)] $\mathfrak{K(v)}/\mathfrak{K(v)}' \cong \mathbb{Z}^s$;
    \item[(v)] $\mathfrak{K(v)}$ is torsion-free;
    \item[(vi)] $\mathfrak{K(v)}$ is contracting;
    \item[(vii)] $\mathfrak{K(v)}$ has exponential word growth.
\end{itemize}
\end{theorem}
\noindent  The above results, apart from part (ii), was known for the case $m=2$; see \cite[Section~3]{BN08}. We note that the fact that the groups $\mathfrak{K(v)}$ are contracting already follows from \cite{BN03},
but we give a self-contained proof which yields more information about the contracting property of the groups. The property of exponential word growth will not be utilised in this paper, so we instead refer the reader to \cite[Chapter~10]{Handbook} for the definition.

In this paper, we study the maximal subgroups of  
$\mathfrak{K(v)}$ for constant words ${\mathfrak{v}}$;
the case of non-constant words prove to be much more difficult and would require a different approach.
The study of maximal subgroups of branch groups was initiated by Pervova~\cite{Per00}, where she proved that the torsion Grigorchuk groups do not contain maximal subgroups of infinite index. Thenceforth, attempts have been made to generalise the results and techniques from \cite{Per00}, for instance see~\cite{AKT13},~\cite{KT18},  and \cite{Fra20}. Among which, our interest lies in the work of Francoeur~\cite{Fra20} (or see \cite[Section 8.4]{Fra19}), who provided a strategy to study the maximal subgroups of weakly branch groups. 
 In particular, he proved that the classical Basilica group does not contain maximal subgroups of infinite index. Following this technique we prove that 
 $\mathfrak{K(v)}$, for constant words ${\mathfrak{v}}$, do not admit maximal subgroups of infinite index.
 
\begin{theorem} \label{thm main}
  Let $m,s\ge 2$    be positive integers, and let ${\mathfrak{v}}=t\,\overset{s-1}{\dots}\,t$ be the constant word for some $t\in {\{0,1,\ldots,m-1\}}$. 
  Then  the group $\mathfrak{K(v)}$
  does not admit a maximal subgroup of infinite index.
\end{theorem}

 \noindent Due to the different properties of the family $\mathfrak{K(v)}$, 
 the final stages of our proof differ from previously seen results; compare \cref{Theorem}{prop prodense}.  This is also the first time that maximal subgroups of a weakly branch, but not branch, group~$G$ have been considered for a group~$G$ with more than $2$ generators.
 
 It is interesting to note that there are currently no examples of finitely generated weakly branch, but not branch, groups  with maximal subgroups of infinite index. There are only examples of finitely generated branch groups with maximal subgroups of infinite index; see \cite{Bondarenko,FG, KS}. It remains to be seen whether being a finitely generated weakly branch group with maximal subgroups of infinite index implies the group is branch.

 Furthermore, in all known examples of finitely generated weakly branch, but not branch, groups with maximal subgroups only of finite index, these groups have maximal subgroups that are not normal; compare Remark~\ref{rem:normal} and \cites{FT, DNT}. Therefore it is also natural to ask if there exists a finitely generated weakly branch, but not branch, group with all maximal subgroups of finite index and normal.

\smallskip

\noindent\textit{Organisation}. \cref{Section}{sec:pre} contains preliminary material on groups acting on the $m$-adic tree, and \cref{Section}{sec: basic properties} establishes basic properties of the groups $\mathfrak{K(v)}$, that is, we prove Theorem~\ref{thm:properties}. In \cref{Section}{sec: length reduction}, we record some length reducing properties of 
a large subfamily of the groups $\mathfrak{K(v)}$,
and in \cref{Section}{sec: maximal subgroups} we prove \cref{Theorem}{thm main}.

\subsection*{Acknowledgements}
We are grateful to the referee  for their valuable comments and for suggesting the iterated monodromy groups, and we thank G.\,A. Fern\'{a}ndez-Alcober and B.~Klopsch for their  helpful feedback.

\smallskip


\section{Preliminaries}\label{sec:pre}

 By $\mathbb{N}$ we denote the set of positive integers, and by $\N$ the set of non-negative integers.

  Let $m\in\mathbb{N}_{\ge 2}$ and let $T = T_m$ be the $m$-adic  tree,
  that is, a rooted tree where all vertices have $m$~children. Using the
  alphabet $X = \{0,1,\ldots,m-1\}$, the vertices~$u_\omega$ of~$T$ are
  labelled bijectively by the elements~$\omega$ of the free
  monoid~$X^*$ in the following natural way:  the root of~$T$
  is labelled  by the empty word, and is denoted by~$\epsilon$, 
  and for each word
  $\omega \in X^*$ and letter $x \in X$ there is an edge
  connecting $u_\omega$ to~$u_{\omega x}$.  More generally, we say
  that $u_\omega$ precedes $u_\lambda$ whenever $\omega$ is a prefix of $\lambda$.

  There is a natural length function on~$X^*$, which is defined as follows: the words~$\omega$ of length $\lvert \omega \rvert = n$, representing vertices
  $u_\omega$ that are at distance~$n$ from the root, are the $n$th
  level vertices and constitute the \textit{$n$th layer} of the tree.

  We denote by $T_u$ the full rooted subtree of~$T$ that has its root at
  a vertex~$u$ and includes all vertices succeeding~$u$. For any
  two vertices $u = u_\omega$ and $v = u_\lambda$, the map
  $u_{\omega \tau} \mapsto u_{\lambda \tau}$, induced by replacing the
  prefix $\omega$ by $\lambda$, yields an isomorphism between the
  subtrees $T_u$ and~$T_v$.  

  Now each $f\in \Aut T$ fixes the root, and the orbits of
  $\Aut T$ on the vertices of the tree~$T$ are the layers of the tree~$T$.  
  The image of a vertex~$u$ under
  $f$ will be denoted by~$uf$. 
  The automorphism~$f$ induces a faithful action
  on~$X^*$ given by
  $(u_\omega)f = u_{\omega f}$.  For $\omega \in X^*$ and
  $x \in X$ we have $(\omega x)f = (\omega f) x'$, for $x' \in X$ 
  uniquely determined by $\omega$ and~$f$. This induces a permutation~$f^\omega$ of~$X$ which satisfies
  \[
  (\omega x) f= (\omega f)xf^\omega, \quad \text{and consequently}
  \quad   (u_{\omega x})f = u_{(\omega f) xf^\omega} .
  \]

   More generally, for an automorphism~$f$ of~$T$, since the layers are invariant under~$f$, for $u\in X^*$, the equation
    \[
    (uv)f=(uf)vf_u \qquad{\text{for every }v\in X^*,}
    \]
    defines a unique automorphism~$f_u$ of~$T$ called the \emph{section of $f$ at $u$}. This automorphism can be viewed as the automorphism of~$T$ induced  by $f$ upon identifying the rooted subtrees of~$T$ at the vertices $u$ and $uf$ with the tree~$T$. As seen here,  we often do not differentiate between $X^*$ and vertices of~$T$. 

 \subsection{Subgroups of $\Aut T$}\label{subsec:2.1}

 Let $G$ be a subgroup of $\Aut T$ acting \textit{level-transitively}, that is, transitively on every layer of~$T$. The
 \textit{vertex stabiliser} $\st_G(u)$ is the subgroup
 consisting of elements in~$G$ that fix the vertex~$u$.  For
 $n \in \mathbb{N}$, the \textit{$n$th level stabiliser}
  $\mathrm{St}_G(n)= \bigcap_{\lvert \omega \rvert =n}
  \st_G(u_\omega)$
 is the subgroup consisting of automorphisms that fix all vertices at
 level~$n$.  

 Each $g\in \mathrm{St}_{\Aut T} (n)$ can be 
 completely determined in terms of its restrictions   to the subtrees rooted at vertices at level~$n$.  There is a natural isomorphism
   \[
     \psi_n \colon \mathrm{St}_{\Aut T}(n) \longrightarrow
     \prod\nolimits_{\lvert \omega \rvert = n} \Aut T_{u_\omega}
     \cong \Aut T \times \overset{m^n}{\cdots} \times
     \Aut T
   \]
 defined by sending $g\in \mathrm{St}_{\Aut T} (n)$ to its tuple of $m^n$ sections. For conciseness, we will omit the use of $\psi_1$, and simply write $g=(g_1,\ldots,g_m)$ for $g\in\text{St}_{\Aut T}(1)$.

 Let  $\omega\in X^n$ be of length $n$. We further define 
 \[
  \varphi_\omega :\st_{\Aut T}(u_\omega) \longrightarrow \Aut T_{u_\omega} \cong \Aut T
 \]
 to be the map sending $f\in \st_{\Aut T}(u_\omega)$ to the section~$f_{u_\omega}$.

 A group $G \leq \Aut T$ is said to be \emph{self-similar}  if for all $f\in G$ and  all $\omega\in X^*$ the  section~$f_{u_{\omega}}$ belongs to~$G$.  We will denote $G_{\omega}$ to be the subgroup $\varphi_{\omega}(\st_G(u_{\omega}))$.

Recall that for a group $G$ generated by a finite symmetric subset~$S$
(i.e.\ a set for which $S = S^{-1}$),  for every $g \in G$, the \textit{length} of~$g$ with respect to $S$ is
\[
|g| = \min\{n \geq 0 \mid g = s_{1} \cdots s_{n}, \text{ for }s_{1}, \dots, s_{n} \in S\}.
\]
Now assume that $G$ is a self-similar subgroup of $\Aut T$.
Then for every $g\in G$ and every $n\in\N\cup\{0\}$, let
\[
\ell_n(g) = \max \{ |g_u| \mid |u|=n \}.
\]
If there exist $\lambda<1$ and $C,L\in\mathbb{N}$ such that
\[
\ell_n(g) \le \lambda |g|+C,
\quad
\text{for every $n>L$ and every $g\in G$,}
\]
then we say that the group $G$ is \emph{contracting} with respect to $S$.

Let $G$ be a subgroup of $\Aut T$ acting level-transitively. Here the vertex stabilisers at every level are conjugate under~$G$.  We say that the group $G$ is \textit{fractal} if $G_{\omega}:=\varphi_\omega(\st_G(u_\omega))=G$ for every $\omega\in X^*$, after the natural identification of subtrees. Furthermore we say that the group~$G$ is \emph{super strongly fractal} if, for each $n\in \mathbb{N}$, we have $\varphi_\omega(\text{St}_G(n))=G$ for every word  $\omega\in X^n$ of length~$n$.

 The \textit{rigid vertex stabiliser} of $u$ in $G$ is the subgroup
$\rst_G(u)$ consisting of all automorphisms in~$G$ that fix
 all vertices of~$T$ not succeeding~$u$. The \textit{rigid $n$th level stabiliser} is the direct product of the rigid vertex stabilisers of the vertices at level~$n$:
  \[
    \mathrm{Rist}_G(n) = \prod\nolimits_{\lvert \omega \rvert = n}
   \rst_G(u_\omega) \trianglelefteq G.
  \]
    
 We recall that a level-transitive group~$G$ is a \emph{branch group} if $\mathrm{Rist}_G(n)$ has finite index in~$G$ for every $n \in \mathbb{N}$; and $G$ is \emph{weakly branch} if $\mathrm{Rist}_G(n)$ is non-trivial for every $n \in \mathbb{N}$. If, in addition,  the group $G$ is self-similar
  and there exists a subgroup $1 \ne K \leq G$ with
 $K\times \overset{m}\cdots \times K \subseteq {\psi_1}(\St_K(1))$ and
 $\lvert G : K \rvert < \infty$, then $G$ is said to be \emph{regular
 branch over~$K$}. If in the previous definition the condition $\lvert G : K \rvert < \infty$ is omitted, then $G$ is said to be \emph{weakly regular branch over}~$K$.

 \subsection{A basic result} Here we record  a general result that will be useful in the sequel. For $g\in \Aut T$, recall that $g^\epsilon$ denotes the  action induced by~$g$ at the root of~$T$.

 \begin{lemma}\label{lemma commutator product}
 For a self-similar group $G\le \Aut T$, let $z = (z_0,\dots,z_{m-1})z^\epsilon \in G'$. Then $z_0\cdots z_{m-1}\in G'$.
 \end{lemma}

\begin{proof}
  It suffices to prove the result for a basic commutator $[g,h]$, where $g,h\in G$. Write $g=(g_0,\ldots,g_{m-1})g^\epsilon$ and  $h=(h_0,\ldots,h_{m-1})h^\epsilon$. For notational convenience, let us write $\tau=(g^\epsilon)^{-1}$ and $\kappa=(h^\epsilon)^{-1}$, and for $\alpha\in \text{Sym}(X)$ and $x\in X$ we write $x^\alpha$ for the image $\alpha(x)$ of $x$ under $\alpha$. As
  \begin{align*}
   &[g,h]\\
   &= \tau(g_0^{-1},\dots,g_{m-1}^{-1}) \kappa(h_0^{-1},\dots,h_{m-1}^{-1}) (g_0,\dots,g_{m-1})g^\epsilon(h_0,\dots,h_{m-1})h^\epsilon\\ 
    & =  (g_{0^\tau}^{-1},\dots,g_{(m-1)^\tau}^{-1}) (h_{0^{\tau\kappa}}^{-1},\dots,h_{(m-1)^{\tau\kappa}}^{-1}) (g_{0^{\tau\kappa}},\dots,g_{(m-1)^{\tau\kappa}})(h_{0^{\tau\kappa{g^\epsilon}}},\dots,h_{(m-1)^{\tau\kappa{g^\epsilon}}})\tau\kappa{g^\epsilon h^\epsilon},
  \end{align*}
  the result follows. 
\end{proof}


 \section{Properties of $\mathfrak{K(v)}$}\label{sec: basic properties}

 For any two integers $i,j,$ let $\intseg{i}{j}$ denote the set $\{i,i+1,\dots,j-1,j\}$. 
 In the following sections, we fix  
 $m,s\in\mathbb{N}_{\geq 2}$. 
 Let $X = \{0,1,\dots,m-1\}$ and let ${\mathfrak{v}}= x_0\cdots x_{s-2}$ be a word in~$X^*$. Recall that the group $\mathfrak{K(v)}$ is generated by the elements $a_0,\dots,a_{s-1}$, where  
 \begin{align*}
     a_0 = (1,\dots,1,a_{s-1}) \sigma, {\qquad\text{and}\qquad a_{i+1} = (1,\overset{x_i-1}{\ldots},1,                                   a_{i}, 1,\ldots,1),}
\end{align*}
for $i \in \intseg{0}{s-2}.$

To prove Theorem~\ref{thm:properties}(i), we use the following characterisation of iterated monodromy groups by Nekrashevych~\cite[{Theorem}~6.10.8]{Nekrashevych} according to the notation given in \cite[{Theorem}~5.8]{Jones}:
\begin{theorem}
    \label{thm:Jones}
    A subgroup $G\le\text{Aut } T$ is isomorphic to a standard action of the iterated monodromy group of a post-critically finite polynomial if and only if $G$ is conjugate in $\text{Aut }T$ to a group generated by a finite invertible automaton $A$ with the following properties:
    \begin{enumerate}
        \item [(1)] For each non-trivial $a\in A$, there is a unique $b\in A$ and $x\in X$ with ${b_x}=a$.
        \item [(2)] For each $a\in A$ and each cycle $(x_1\,x_2\,\cdots\,x_n)$ of the action of $a$ on $X$, the section $a_{x_i}$ is non-trivial for at most one~${x_i}$.
         \item [(3)] The multi-set of permutations defined by the set of states of $A$ acting on $X$ is tree-like.
          \item [(4)] Let $a_1\ne a_2$ be non-trivial states of $A$ with $\epsilon\ne v_1,v_2\in T$ satisfying $(a_i)_{v_i}=a_i$ and $a_i(v_i)=v_i$ for $i\in\{1,2\}$. Then there is no $h\in G$ with $h(v_1)=v_2$ and $h_{v_1}=h$.
    \end{enumerate}
\end{theorem}

We refer the reader to \cite{Nekrashevych} or \cite{Jones} for any unexplained terminology.
\begin{proof}[Proof of Theorem~\ref{thm:properties}(i)]
   We will show that the above four conditions hold for $\mathfrak{K(v)}$. Conditions (1) and (2) of Theorem~\ref{thm:Jones} are clear from the definition of the generators $a_i$. For condition (3), notice that $a_0$ is the only generator of $\mathfrak{K(v)}$ that acts non-trivially on $X$. Therefore, according to \cite[Section~5.4]{Jones}, the multi-set of permutations of $X$ is given by $\{\sigma\}$. Furthermore, the cycle diagram associated to the multi-set $\{\sigma\}$  is a oriented 2-dimensional CW-complex, whose set of 0-cells is $X$, and  with one 2-cell obtained by connecting the elements of $X$ by the action of~$\sigma$. Hence the cycle diagram is contractible. Thus, by \cite[Definition~5.7]{Jones} the multi-set of permutations defined by the set of states of $\mathfrak{K(v)}$ acting on~$X$ is tree-like.
    
    For condition (4), the restriction $(a_i)_{v_i}=a_i$ implies that the $v_i$ have to be at level $ns$ for some positive integer~$n$. However, for any $a_i$, after $i$ levels, the section is $a_0$, which does not then fix any vertex below it. So  we can never find vertices $v_i$ such that $a_i(v_i)=v_i$. So (4) vacuously holds.
\end{proof}

\begin{lemma}\label{lemma_k(v)_fractal_and_level-trans}
    The group $\mathfrak{K(v)}$ is super strongly fractal and level-transitive.
\end{lemma}

\begin{proof}
    Observe that $a_0^m=(a_{s-1},\dots,a_{s-1})$ and, for any $i \in \intseg{1}{s-1}$ and $j \in \intseg{0}{m-1}$, 
    \begin{align*}
        {\varphi_{j}\bigg(}a_i^{a_0^{j-x_{i-1}}}{\bigg)}=a_{i-1}.
    \end{align*}
 It follows that ${\varphi_x(\St}_{\mathfrak{K(v)}}(1))=\mathfrak{K(v)}$ for all $x \in X$. For $n\in \mathbb{N}$, we write $n=\ell s+r$, where $\ell\in \mathbb{N}_0$ and $r\in[0,s-1]$.  Using similar arguments with the elements 
    \[
    a_0^{m^{\ell+1}}, a_1^{m^{\ell+1}}, \ldots,  a_{r-1}^{m^{\ell+1}}, a_{r}^{m^{\ell}},\ldots,  a_{s-1}^{m^{\ell}}\in \St_{\mathfrak{K(v)}}(n),
    \]
    and their conjugates, we deduce that ${\varphi_w(\St}_{\mathfrak{K(v)}}(n))=\mathfrak{K(v)}$ for all $w \in X^n$. 
    
     The second statement is immediate from the fact that ${\varphi_x(\st}_{\mathfrak{K(v)}}(x))=\mathfrak{K(v)}$ for some $x \in X$, and since $a_0$, and hence $\mathfrak{K(v)}$, acts transitively on the first layer.
\end{proof}

Next we record an elementary but useful result. Recall that ${\mathfrak{v}} = x_0\cdots x_{s-2}$ is a word in~$X^*$.

\begin{lemma}\label{lem:commuting}
 For distinct $i,j\in[1,s-1]$ with $x_{i-1}\ne x_{j-1}$, we have $[a_i,a_j]=1$ in~$\mathfrak{K(v)}$. For $i>\ell\in[1,s-1]$ with $x_{i-1}= x_{\ell-1}$, we have $[a_i,a_\ell]=1$ if
$x_{i-d}\ne x_{\ell-d}
 $ for some $d\in[2,\ell]$  and $[a_i,a_\ell]\ne 1$ otherwise.
\end{lemma}

\begin{proof}
The first statement is a straightforward computation. For the next statement, note  that for  $i\in[1,s-1]$,
 \[
 [a_i,a_0]=\begin{cases}
     (a_{s-1}^{-1}a_{i-1}a_{s-1},1,\ldots,1,a_{i-1}^{-1})&\text{if }x_{i-1}= m-1,\\
      (1,\overset{x_{i-1}-1}\ldots,1,a_{i-1}^{-1},a_{i-1},1,\ldots,1)&\text{if }x_{i-1}\ne m-1,
 \end{cases}
 \]
 so $a_0$ does not commute with any other $a_i$. The result then follows, using the fact that 
 \[
 [a_i,a_\ell]=(1,\ldots,1,[a_{i-1},a_{\ell-1}],1,\ldots,1).\qedhere
 \]
\end{proof}

\begin{lemma} \label{lemma_k(v)_weakly_branch}
    The group $\mathfrak{K(v)}$ is weakly regular branch over $\mathfrak{K(v)}'$. 
\end{lemma}

\begin{proof}
    We use that $\mathfrak{K(v)}' = \langle [a_i,a_j] \mid i,j \in \intseg{0}{s-1}\rangle^{\mathfrak{K(v)}}$. We have
    \begin{align*}
        \Big[a_i^{a_0^{m-1-x_{i-1}}}, a_j^{a_0^{m-1-x_{j-1}}}\Big] &= (1,\dots,1,[a_{i-1},a_{j-1}]),\\
        [a_i, a_0^m]^{a_0^{m-1-x_{i-1}}} &= (1,\dots,1,[a_{i-1}, a_{s-1}]),
    \end{align*}
 for $i,j\in \intseg{1}{s-1}.$  Therefore, by transitivity and fractalness, we have $\mathfrak{K(v)}' \times \overset{m}\cdots \times \mathfrak{K(v)}' \leq \psi(\mathfrak{K(v)}')$. 
\end{proof}

For $i \in \intseg{0}{s-1}$, let $A_i = \langle a_j\mid {j\in[0,s-1]\backslash\{i\}} \rangle^{\mathfrak{K(v)}}$. 

\begin{lemma}\label{lemma_order_of_a_i}
    For $i \in \intseg{0}{s-1}$, 
    we have
    $\mathfrak{K(v)}/A_i \cong \mathbb{Z}$. In particular, the elements~$a_i$ have infinite order in $\mathfrak{K(v)}$.
\end{lemma}

\begin{proof}
    We  
    prove  simultaneously for all $i$ that $\mathfrak{K(v)}/A_i \cong \mathbb{Z}.$ Assume for a contradiction that, for some $n \in \mathbb{N}$ that $a_i^n \in A_i$, for some $i \in \intseg{0}{s-1}.$ Choose $n \in \mathbb{N}$ minimal with respect to the property that $a_i^n \in A_i$ for some $i \in \intseg{0}{s-1}$. If $i =0$, then $a_{0}^n \in \St_{\mathfrak{K(v)}}(1)$ and, in particular $n \equiv 0 {\pmod m}$. We have 
      \begin{align*}
          a_0^n = (a_{s-1}^{\nicefrac{n}{m}}, \dots, a_{s-1}^{\nicefrac{n}{m}}),
      \end{align*}
    and hence $a_{s-1}^{\nicefrac{n}{m}} \in A_{s-1}.$ This is a contradiction to the minimality of $n$. Now suppose that $i \neq 0$. Then by considering appropriate sections  
    of $a_i^n$, we see that $a_0^n \in A_0$, which cannot happen as {shown} above. Therefore, we conclude that ${\mathfrak{K(v)}}/A_i \cong \mathbb{Z}$ for all $i \in \intseg{0}{s-1}$.
\end{proof}

\begin{lemma}
    \label{lem:not-branch}
     We have  
     $\textup{Rist}_{\mathfrak{K(v)}}(1)=A_0$. Hence the group $\mathfrak{K(v)}$ is not branch.
\end{lemma}

\begin{proof}
Analogous to the proof of \cite[Theorem 3.5(i)]{DNT},
we clearly have $A_0\le \textup{Rist}_{\mathfrak{K(v)}}(1)\le \St_{\mathfrak{K(v)}}(1)=A_0\langle a_0^m \rangle$, where the last equality follows from the fact that $a_i\in\St_{\mathfrak{K(v)}}(1)$ for all $i\ne 0$, and that $a_0^n\in\St_{
\mathfrak{K(v)}}(1)$ if and only if $n\equiv 0\pmod m$. So
$\textup{Rist}_{\mathfrak{K(v)}}(1)=A_0\langle a_0^{mn} \rangle$ for some~$n$.
Since 
\[
\psi(A_0\langle a_0^{mn} \rangle)=(A_{s-1}\times \cdots \times A_{s-1})\langle (a_{s-1}^n,\ldots,a_{s-1}^n) \rangle
\]
and $a_{s-1}$ has infinite order modulo $A_{s-1}$ by \cref{Lemma}{lemma_order_of_a_i}, it follows from the definition of the rigid stabiliser that $n=0$.
Hence $\textup{Rist}_{\mathfrak{K(v)}}(1)=A_0$, which has infinite index in $\mathfrak{K(v)}$, so $\mathfrak{K(v)}$ is not branch.
\end{proof}

 We shall adopt the convention that the subscripts of the $a_i$'s are taken modulo~$s$.
 Set $S = \{a_i^{\pm 1} \mid i \in \intseg{0}{s-1}\}$ and then $\mathfrak{K(v)}= \langle S \rangle$.
 For each word $w \in S^*$, the length $|w|$ is the usual word length of $w$ over the alphabet~$S$. 
\begin{lemma}\label{lemma_exp_sum}
    Let $w$ be a word in~$S$ 
    representing the identity in $\mathfrak{K(v)}$. Then the exponent sum of $a_i$ in $w$ must be zero for all $i \in \intseg{0}{s-1}$.
\end{lemma}

\begin{proof}
   We proceed by induction on the length of $w$. Let $w$ be a non-trivial word in~$S$ 
    representing the identity in $\mathfrak{K(v)}$. It is immediate from \cref{Lemma}{lemma_order_of_a_i} that $|w| \geq 2$ and $w$ must contain non-trivial powers of at least two distinct $a_i,a_j \in S$. For the base case $|w|=2$, we have that $w=a_i^{\epsilon_i}a_j^{\epsilon_j}$ for distinct $i,j\in [0,s-1]$ and $\epsilon_i,\epsilon_j\in\{\pm 1\}$. From considering the image of~$w$ in $\mathfrak{K(v)}/A_i$, we obtain a contradiction to $\epsilon_i$ being non-zero.

   Assume that the result holds for all words of length $n \geq 2$, and let $|w|=n+1$. By realising~$w$ in $\mathfrak{K(v)}$, we can see that the exponent sum of $a_0$ in~$w$ must be zero modulo~$m$. 
    By abuse of notation, we write
    \[\psi(w) = (w_0,\dots,w_{m-1}),\]
    where~$w_i$ are reduced words determined by the appropriate sections of the letters of~$w$. Since~$w$ represents the identity in $\mathfrak{K(v)}$, each~$w_k$ also represents the identity in $\mathfrak{K(v)}$. It follows from the definition of the generators $a_i$ that $\sum \limits_{k=0}^{m-1}|w_{k}| \leq |w|$. In particular, if $|w_{k}| = |w|$, for some $k \in \intseg{0}{m-1}$, then the $w_j$ are trivial for all $j \in \intseg{0}{m-1}\backslash\{k\}$, otherwise $|w_{k}| < |w|$ for all ${k} \in \intseg{0}{m-1}$. 
    
    Assume that $|w_{k}| < |w|$ for all $k \in \intseg{0}{m-1}$. By the induction hypothesis, for all $i \in \intseg{0}{s-1}$ the exponent sum of~$a_i$ is zero in $w_{k}$ for all  ${k} \in \intseg{0}{m-1}$. Since the~$a_i$ in $w_{k}$ are obtained from~$a_{i+1}$ in $w$, we conclude that the exponent sum of~$a_{i+1}$ in $w$ is zero for all $i \in \intseg{0}{s-1}$. 

    Now, assume that there exists $k \in \intseg{0}{m-1}$ such that $|w_k| = |w|$. By replacing~$w$ with~$w_k$, repeatedly if necessary, we may assume 
    that $w$ does  
    involve the letters $a_0^{\pm 1}$. 
   Then, if $w$ has $a_0^{\pm 1}$ as a letter, then the total exponent of~$a_0$ is congruent to~0 modulo~$m$. Suppose first that the total exponent of $a_0$ is non-zero. 
 Since the total exponent of $a_0$  is congruent to $0$ modulo~$m$, we can find a subword $a_0^{\pm 1}w' a_0^{\pm 1}$ such that the exponent sum of $a_0$ in $w'$ is zero. By realising this subword in $\mathfrak{K(v)}$, we immediately see a length reduction among the sections, and it follows by induction that every section $w_k$ has exponent sum of $a_i$ equal to zero for all $i\in [0,s-1]$. So in particular, the case  that total exponent of $a_0$ in $w$ is non-zero cannot occur.     So assume that the total exponent of $a_0$ is zero. Upon considering each $w_k$, we may  assume that the total exponent of $a_0$ in~$w_k$ is  zero, else we are done by the above argument. Hence the total exponent of $a_1$ in~$w$ is zero. Recursively, we deduce 
    that the total exponent in~$w$ of any $a_i$ is zero. The result then follows. 
\end{proof}

We proceed to prove parts (iv) and (v) of \cref{Theorem}{thm:properties}.

\begin{lemma}\label{lemma_abelian}
    The quotient group $\mathfrak{K(v)}/\mathfrak{K(v)}'$ is isomorphic to the free abelian group of rank $s$.
\end{lemma}

\begin{proof}
    For every $g \in \mathfrak{K(v)}$, there exist integers $n_0,\dots,n_{s-1}$ such that $g \equiv a_0^{n_0}\cdots a_{s-1}^{n_{s-1}}$ modulo~$\mathfrak{K(v)}'$. It is enough to show that, if $a_0^{n_0}\cdots a_{s-1}^{n_{s-1}} \in \mathfrak{K(v)}'$, for some $n_0,\dots,n_{s-1} \in \mathbb{Z}$, then $n_0=\cdots=n_{s-1} = 0$. 

    Assume that $a_0^{n_0}\cdots a_{s-1}^{n_{s-1}} \in \mathfrak{K(v)}'$ for some $n_0,\dots,n_{s-1} \in \mathbb{Z}$. It is clear that every element in $\mathfrak{K(v)}'$ can be written as a word in~{$S$} 
    and the exponent sum of each~$a_i$ is zero. It follows from \cref{Lemma}{lemma_exp_sum} that the exponent sums of the $a_i$'s in all words representing the same element in $\mathfrak{K(v)}$ are the same. This, in particular, applies  to $a_0^{n_0}\cdots a_{s-1}^{n_{s-1}}$, and hence $n_0 = \cdots = n_{s-1} = 0$.     
\end{proof}

\begin{lemma}\label{lemma_k(v)_torsion}
   The group $\mathfrak{K(v)}$ is torsion-free.
\end{lemma}
\begin{proof}
  It follows immediately from \cite[Theorem 3.6]{DNT} that $\mathfrak{K(v)}$ is torsion-free as the quotient group $\mathfrak{K(v)}/\mathfrak{K(v)}'$ is torsion-free and the subgroup $\mathfrak{K(v)}'$ is contained in $\St_{\mathfrak{K(v)}}(1)$. 
\end{proof}

Next, we finish the proof of \cref{Theorem}{thm:properties}. Recall that for $g\in\mathfrak{K(v)}$, the $n$th level sections $g_0,\ldots,g_{m^n-1}$ of~$g$ are determined from expressing $g$ as the product $$
g=\psi_n^{-1}((g_0,\ldots,g_{m^n-1}))\tau_g,
$$
where $\psi_n^{-1}((g_0,\ldots,g_{m^n-1}))\in\St_{\mathfrak{K(v)}}(n)$ and $\tau_g\in\Sym(X^n)$. 

\begin{theorem}
\label{thm: contracting}
The group $\mathfrak{K(v)}$  is contracting with respect to the set of generators
$S$, with $\lambda=\frac{2}{3}$, $L=s-1$ and $C=1$.
\end{theorem}

\begin{proof}
By \cite[Lemma~3.7]{DNT}, it suffices to prove that
\begin{equation}
\label{contracting level 2}
\ell_s(g)\le \frac{2}{3}|g|+1, \quad \text{for every $g\in \mathfrak{K(v)}$.}
\end{equation}
We can further reduce the problem to showing 
that $\ell_s(h)\le 2$ for every $h\in \mathfrak{K(v)}$ of length $3$.
Indeed, writing 
$|g|=3\alpha+\beta$ 
with $\beta\in\{0,1,2\}$ and
$g=h_1\cdots h_\alpha f$ with $|h_1|=\cdots=|h_\alpha|=3$ and $|f|=\beta$, we see that  (\ref{contracting level 2}) immediately follows from the subadditivity of $\ell_s$; cf. \cite[Equation~(3.3)]{DNT}.

Let us then consider an arbitrary element $h\in \mathfrak{K(v)}$ of length $3$ and prove that $\ell_s(h)\le 2$. 
Observe that $h$ is of the form 
\begin{enumerate}
    \item [(i)] $a_i^{\epsilon_i}a_j^{\epsilon_j}a_k^{\epsilon_k}$ for $i\ne j$ and $j\ne k$; or
     \item [(ii)] $a_i^{2\epsilon_i}a_j^{\epsilon_j}$ or $a_i^{\epsilon_i}a_j^{2\epsilon_j}$ for $i\ne j$; or
      \item [(iii)] $a_i^{3\epsilon_i}$,
\end{enumerate}
where $\epsilon_i, \epsilon_j, \epsilon_k\in\{\pm 1\}$. 
For case (iii), it is clear that $\ell_{i+1}(h)\le 2$, since the only non-trivial $i$th level section is $a_0^{3\epsilon_i}$. Similarly for case (ii). 

For case (i), suppose 
for a contradiction that $\ell_n(h)=3$ for all $n\in\mathbb{N}$. 
Then note that for every~$n$, the element~$h$ has exactly one non-trivial $n$th level section, which is $a_{i-n}^{\epsilon_i}a_{j-n}^{\epsilon_j}a_{k-n}^{\epsilon_k}$. As usual, the indices of the generators are viewed modulo~$s$, so we effectively only consider the first $(s-1)$st level sections.

Consider now the non-trivial $i$th level section, which is $a_0^{\epsilon_i}a_{j-i}^{\epsilon_j}a_{k-i}^{\epsilon_k}$.
Suppose $k=i$. If $\epsilon_i=\epsilon_k$, then clearly the $(i+1)$st sections decrease in length. So suppose $\epsilon_i\ne \epsilon_k$. 
Upon considering the only non-trivial $j$th level section $a_{i-j}^{\epsilon_i}a_{0}^{\epsilon_j}a_{i-j}^{\epsilon_k}$, we see that $\ell_{j+1}(h)\le 2$. So we suppose that $k\ne i$. Consider first the case 
$\epsilon_i=1$. Recall that the non-trivial $i$th level section of $h$ is $a_0a_{j-i}^{\epsilon_j}a_{k-i}^{\epsilon_k}$. In order for  $\ell_{i+1}(h)$ to still be 3, we need 
\[
a_{j-i}=(a_{j-i-1},1,\ldots,1)\qquad\text{and}\qquad a_{k-i}=(a_{k-i-1},1,\ldots,1).
\]
If the defining word~${\mathfrak{v}}$ of the group~$\mathfrak{K(v)}$ does not correspond to the above generators, then this implies that our assumption for this case (i) does not hold, and hence the section lengths eventually decrease, giving us $\ell_s(h)\le 2$. So we proceed by assuming that ${\mathfrak{v}}$ tallies with the above generators and with the other generators singled out below.

Next, we similarly see that in order for  $\ell_{s+2i-j}(h)$ to still be 3, we need $a_{j-i}a_{2j-2i}^{\epsilon_j}a_{j+k-2i}^{\epsilon_k}$ to have exactly one non-trivial section. Thus,
\[
a_{2j-2i}=(a_{2j-2i-1},1,\ldots,1)\qquad\text{and}\qquad a_{j+k-2i}=(a_{j+k-2i-1},1,\ldots,1).
\]
In particular, proceeding in this manner, we may assume that
\[
a_{\theta(j-i)}=(a_{\theta(j-i)-1},1,\ldots,1)
\]
for all $\theta\in\mathbb{N}$. Now let $\theta=s-1$, and so $a_{i-j}=(a_{i-j-1},1,\ldots,1)$. Observe that the only non-trivial $j$th level section of $h$ is $a_{i-j}a_0^{\epsilon_j}a_{k-j}^{\epsilon_k}$. To avoid a decrease in length at the next level, we must have $\epsilon_j=-1$ and 
$a_{k-j}=(1,\ldots,1,a_{k-j-1})$. Similarly, 
 for  $\ell_{s+2j-k}(h)$ to still be 3, we need $a_{k-2j+i}a_{k-j}^{-1}a_{2(k-j)}^{\epsilon_k}$ to have exactly one non-trivial section. Thus,
\[
a_{k-2j+i}=(1,\ldots,1,a_{k-2j+i-1})\qquad\text{and}\qquad a_{2(k-j)}=(1,\ldots,1,a_{2(k-j)-1}).
\]
As deduced above, since $a_{k-i}=(a_{k-i-1},1,\ldots,1)$ we also then have that 
\[
a_{\theta(k-i)}=(a_{\theta(k-i)-1},1,\ldots,1)\qquad\text{and}\qquad a_{\theta(k-j)}=(1,\ldots,1,a_{\theta(k-j)-1})
\]
for all $\theta\in\mathbb{N}$. As before, we let $\theta=s-1$, and consider the only non-trivial $k$th level section of~$h$, which is $a_{i-k}a_{j-k}^{-1}a_{0}^{\epsilon_k}$. Since
\[
a_{(s-1)(k-i)}=a_{i-k}=(a_{i-k-1},1,\ldots,1)\qquad\text{and}\qquad a_{(s-1)(k-j)}=a_{j-k}=(1,\ldots,1,a_{j-k-1})
\]
Then we see that $\ell_{k+1}(h)\le 2$, and we are done in this case $\epsilon_i=1$.

The remaining case 
$\epsilon_i=-1$ follows similarly. Hence the proof is complete.
\end{proof}

\begin{lemma}\label{lemma_k(v)_exp}
    The semigroup generated by $a_0,\dots,a_{s-1}$ is free. In particular, the group $\mathfrak{K(v)}$ has exponential word growth.
\end{lemma}

\begin{proof}
We proceed as in the proof of \cite[Theorem 6.1]{DNT}, which is based on \cite[Lemma 4]{GZ02}. 
Let $u$ and $w$ be two different words representing the same element in the semigroup generated by $a_0,\dots,a_{s-1}$, and with $\rho=\max(|u|,|w|)$ minimal. Clearly $\rho\ge 2$. Now for any word~$z$ in the semigroup, let $|z|_{a_0}$ denote the $a_0$-length in~$z$, i.e. the number of occurrences of~$a_0$ in~$z$.
Note that we have $|u|_{a_0}\equiv |w|_{a_0} \pmod m$. In the rest of the proof we may consider the sections of $u$ and $w$ as defined in the proof of \cref{Lemma}{lemma_exp_sum}.

Suppose first that $u$ contains no $a_0$'s. We may assume that $w$ contains at least one $a_0$. Indeed, if both $u$ and $w$ have no $a_0$'s, by replacing both words with their corresponding non-trivial sections respectively, the condition of both words having no $a_0$'s cannot hold indefinitely, as then both words $u$ and $w$ would be the trivial word. Thus $w$ contains at least one $a_0$. Since $|w|_{a_0}$ must then be a non-zero multiple of~$m$,
one deduces that $w_j$, for $j\in[0,m-1]$, is a non-empty word, where $\psi(w)=(w_0,\ldots,w_{m-1})$. Indeed, every component of $\psi(w)$ contains an~$a_{s-1}$. Certainly $|w_j|<\rho$. Therefore $u$ must have only one section with a non-empty word, as otherwise it will contradict the minimality of $\rho$. Suppose this section is in the $i$th component, for some $i\in[0,m-1]$. From considering the $j$th section of $u$ and $w$, for $j\ne i$, we obtain a contradiction to the minimality of~$\rho$. So the number of occurrences of~$a_0$ in~$u$ is at least one.

If the number of occurrences of $a_0$ in $u$, respectively $w$, is at least 2, then there are at least two sections of $u$, respectively $w$, that contain $a_{s-1}$. Hence all sections of $u$, respectively $w$, have length strictly less than $|u|$, respectively $|w|$, which contradicts the minimality of $\rho$.  So suppose that both $u$ and $w$ have exactly one occurrence of $a_0$. Repeating this argument for the sections of $u$ and $w$, we may assume that every generator appears at most once in $u$ and $w$, and that the same generators appear in both $u$ and  $w$. In other words, the word $w$ is a reordering of the letters in $u$. Furthermore, as seen above, both $u$ and $w$ have only one non-empty section, which is therefore the rightmost section. Indeed, if this is not the case, then either $u$ or $w$ have more than one non-trivial section, then we contradict the minimality of $\rho$. 
Without loss of generality, we may suppose that $a_0$ is the $i$th letter of the word $u$, but the $j$th letter of the word $w$ where $j<i$. The restrictions on the sections of~$u$ implies that the first $i$ letters of $u$ have non-empty section in the rightmost component, and all other letters have non-empty section in the leftmost component. 
At least one of these first $i-1$ letters of $u$ appears to the right of $a_0$ in the word $w$. Thus, there is a non-empty section in the second last component, which yields a contradiction to the minimality of $\rho$. Hence we are done.
\end{proof}

Lastly, we end this section with a result that will be useful for the final section.
\begin{lemma}\label{lemma_direct_product_gamma_3}
    The subgroup $\psi^{-1}(\gamma_3(\mathfrak{K(v)}) \times \overset{m}\cdots \times \gamma_3(\mathfrak{K(v)}))$ is contained in $\mathfrak{K(v)}''$.
\end{lemma}

\begin{proof}
    Notice first that, for $i \in \intseg{1}{s-1}$ with $x_{i-1} \neq m-1$, we have
    \begin{align*}
        [a_i,a_0]^{a_0^{m-2-x_{i-1}}} = (1,\dots,1,a_{i-1}^{-1},a_{i-1}).
    \end{align*}
  When $x_{i-1} = m-1$, we get 
    \begin{align*}
      [a_0^{-1}, a_i] = (1,\dots,1,a_{i-1}^{-1},a_{i-1}).
    \end{align*}
  Together with 
  \cref{Lemma}{lemma_k(v)_weakly_branch}, we see that the elements of the form 
  \begin{align*}
      (1,\dots,1,{[a_i,a_j,a_k]})
  \end{align*}
  are contained in $\psi(\mathfrak{K(v)}'')$, for $i,j \in \intseg{0}{s-1}$ and $k \in \intseg{0}{s-2}$.

  We have to find the elements of the form $(1,\dots,1,{[a_i,a_j,a_{s-1}]} )$. Observe from \cref{Lemma}{lemma_k(v)_weakly_branch} that
  \[
  \mathfrak{K(v)}''\times \overset{m}{\cdots}\times  \mathfrak{K(v)}''\le \psi(\mathfrak{K(v)}'').
  \]
  Then from the Hall-Witt identity, we deduce that
  \begin{align*}
 1&=[a_j,a_i^{-1},a_{s-1}]^{a_i} [a_i,a_{s-1}^{-1},a_j]^{a_{s-1}}[a_{s-1},a_j^{-1},a_i]^{a_j}\\
 &=[a_i,a_j,a_{s-1}^{a_i}] [a_{s-1},a_i,a_j^{a_{s-1}}][a_j,a_{s-1},a_i^{a_j}]\\
 &\equiv [a_i,a_j,a_{s-1}] [a_{s-1},a_i,a_j][a_j,a_{s-1},a_i] \pmod {\mathfrak{K(v)}''}.
  \end{align*}
 Hence, if $i,j\ne s-1$, from considering the product of  $$(1,\dots,1,{[a_{s-1},a_i,a_j] } )\qquad\text{and}\qquad (1,\dots,1,{[a_j,a_{s-1},a_i] } )$$
 together with an appropriate element from $\mathfrak{K(v)}''\times \overset{m}{\cdots}\times  \mathfrak{K(v)}''$, we obtain
 \[
 (1,\dots,1,{[a_i,a_j,a_{s-1}] } 
)\in \psi{(\mathfrak{K(v)}'')}.
 \]
 If $i=s-1$ but $j\ne s-1$, we will show that 
$ {[a_{s-1},a_j} ,a_{s-1}]=1$. To this end, in light of \cref{Lemma}{lem:commuting} it suffices to consider $ {[a_{s-1-j},a_0},a_{s-1-j}]$. Recall from the proof of \cref{Lemma}{lem:commuting} that
\[
 [a_{s-1-j},a_0]=\begin{cases}
     (a_{s-1}^{-1}a_{s-j-2}a_{s-1},1,\ldots,1,a_{s-j-2}^{-1})&\text{if }x_{s-j-2}= m-1,\\
      (1,\overset{x_{s-j-2}-1}\ldots,1,a_{s-j-2}^{-1},a_{s-j-2},1,\ldots,1)&\text{if }x_{s-j-2}\ne m-1.
 \end{cases}
 \]
 Hence it is now clear that $ {[a_{s-1-j},a_0},a_{s-1-j}]=1$.
 
  If $j=s-1$ but $i\ne s-1$,
  from the Hall-Witt identity,  we have the equivalence
 \[
 1\equiv[a_{i},a_{s-1},a_{s-1}] [a_{s-1},a_i,a_{s-1}] \pmod {\mathfrak{K(v)}''}.
 \]
 Hence we are done from the previous case. 
 
 The result now follows from the level-transitivity and the fractalness of $\mathfrak{K(v)}$. 
\end{proof}

\section{Length reducing properties}
\label{sec: length reduction}

 As before, we have ${\mathfrak{v}} = x_0 \cdots x_{s-2}$ is a word in the alphabet $X=\{0,1,\dots,m-1\}$ and $\mathfrak{K(v)}$ is the group associated to ${\mathfrak{v}}$. In this section, we establish some length reducing properties for some of the groups $\mathfrak{K(v)}$.  Notice that for every $g \in \mathfrak{K(v)}$, the local action $g^{\epsilon}$ of $g$ at the root is an element of~$\langle \sigma \rangle$. Hence, for conciseness, we denote $g^{\epsilon}$ by~$\sigma_g$. If $g \in \mathfrak{K(v)}$ then~$|g|$ denotes the minimal length of all words in the alphabet~$S$ representing~$g$. A word $w \in S^*$ is called a \emph{geodesic word} if $\vert w \vert = \vert g \vert$, where $g$ is the image of the word~$w$ in~$\mathfrak{K(v)}$.

 \begin{lemma}\label{lemma length red 1}
  Let 
  $g = (g_0,\dots,g_{m-1})\sigma_g \in \mathfrak{K(v)}.$ Then $\sum \limits _{k =0}^{m-1} \vert g_k \vert \leq \vert g \vert$.
 \end{lemma}

 \begin{proof}
  The proof proceeds by induction on the length of $g$. Clearly, the result is true if $\vert g \vert =0$ and $\vert g \vert =1.$ Assume that $\vert g \vert >1.$ Let $w \in S^*$ be a geodesic word representing $g$. The word~$w$ can be written as $w=bw'$
  for some $b \in S$ and $w'\in S^*$ such that $w'$ is reduced. Then $\vert w' \vert < \vert w \vert$ and $w'$ does not represent $g$ in $\mathfrak{K(v)}$.  Denote by $g'$ the corresponding element in $\mathfrak{K(v)}$. Then $\vert g' \vert \leq \vert w' \vert < \vert w \vert = \vert g \vert$. We obtain 
   \begin{align*}
    (g_0,\dots,g_{m-1})\sigma_{g}= g 
    & = b  g' = (b_0,\dots,b_{m-1})\sigma_{b}(g'_0,\dots,g'_{m-1})\sigma_{g'} \\
    & = ({b}_0g'_{0^{\sigma_{b}}},\dots,{b}_{m-1}g'_{(m-1)^{\sigma_{b}}})\sigma_{b}\sigma_{g'},
  \end{align*}
 which implies $g_k = {b}_kg'_{k^{\sigma_{b}}}$ for all $k \in \intseg{0}{m-1}$. It follows by induction that
  \[
   \sum \limits _{k =0}^{m-1} \vert g_k \vert = \sum \limits _{k =0}^{m-1} \vert {b}_kg'_{k^{\sigma_{b}}} \vert \leq \sum \limits _{k =0}^{m-1} \vert  {b}_k \vert + \sum \limits _{k =0}^{m-1} \vert g'_{k^{\sigma_{b}}} \vert \leq \vert {b} \vert + \vert g' \vert \leq  \vert {b} \vert + \vert w' \vert = \vert w \vert = \vert g \vert. \qedhere
  \]
 \end{proof}

 \begin{lemma}\label{lemma length red 2} 
  Let ${\mathfrak{v}} = x_0 \cdots x_{s-2}$ be a word in the alphabet $X$, and let $g = (g_0,\dots,g_{m-1})\sigma_g \in \mathfrak{K(v)}$ with $\sigma_g = \sigma^{i}$ for some $i \in \intseg{1}{m-1}$ such that $\gcd{(i,m)}=1$. Let $\alpha_0,\dots,\alpha_{m-1} \in \mathfrak{K(v)}$ be such that $g^m = (\alpha_0,\dots,\alpha_{m-1})$. Then $\vert \alpha_k \vert \leq \sum \limits _{\ell =0}^{m-1} \vert g_{\ell} \vert\leq \vert g \vert$ for all $k \in \intseg{0}{m-1}$.
 \end{lemma}

\begin{proof}
 Observe that
 \begin{align*}
    g^m =\big(g_{0}g_{0^{\sigma^i}}g_{0^{\sigma^{2i}}}\cdots g_{0^{\sigma^{(m-1)i}}}\,,\,\dots\,,\,g_{m-1}g_{(m-1)^{\sigma^i}}g_{(m-1)^{\sigma^{2i}}}\cdots g_{(m-1)^{\sigma^{(m-1)i}}}\big).
 \end{align*}
 By setting $\alpha_k = g_{k}g_{k^{\sigma^i}}\cdots g_{k^{\sigma^{(m-1)i}}}$, for each $k \in \intseg{0}{m-1}$, we obtain 
  \[
    \vert \alpha_k \vert = \vert g_{k}g_{k^{\sigma^i}}\cdots g_{k^{\sigma^{(m-1)i}}} \vert \leq \sum \limits _{\ell =0}^{m-1} \vert g_{\ell} \vert \leq \vert g \vert,
  \]
 where the last inequality follows from \cref{Lemma}{lemma length red 1}.
\end{proof}

\begin{lemma}
\label{lemma length red 3}
 Let ${\mathfrak{v}} = x_0 \cdots x_{s-2}$ be a word in the alphabet $X\backslash\{0\}$. Let $g = (g_0,\dots,g_{m-1})\sigma_g \in {\mathfrak{K(v)}}$ and let ${b}_1 \cdots {b}_\ell \in S^*$ be a geodesic word representing~$g$. If there exist $1 \leq r < r' \leq \ell$ such that ${b}_r = a_{0}$, ${b}_{r'}=a_0^{-1}$, then $\sum \limits _{k =0}^{m-1} \vert g_k \vert < \vert g \vert.$ 
\end{lemma}

\begin{proof}
 By assumption, the word ${b}_1\cdots {b}_\ell$ contains a subword of the form $a_{0}wa_{0}^{-1}$, where $w$ is a non-trivial reduced word in the alphabet~$S$. We assume, without loss of generality, that $w$ is a reduced word in the alphabet~$S\setminus \{a_0^{\,\pm 1}\}$. Let $w$ represent an element~$h$ in~$\mathfrak{K(v)}$. Since ${b}_1 \cdots {b}_\ell \in S^*$ is a geodesic word, the word~$w$ is also geodesic and so $|h| = |w|$. Notice that $\vert a_{0}w a_{0}^{-1} \vert = \vert w \vert + 2$. Realising the word $a_{0}wa_{0}^{-1}$ in~$\mathfrak{K(v)}$ gives
 \begin{align*}
  a_{0}wa_{0}^{-1}=(\varphi_{1}(h),\dots,\varphi_{m-2}(h),\varphi_{m-1}(h),1);  
 \end{align*} 
 indeed, recall that by assumption we have 
 \[
   h = (1, \varphi_{1}(h),\dots,\varphi_{m-2}(h),\varphi_{m-1}(h)).
 \]
 By \cref{Lemma}{lemma length red 1}, we get $\sum_{i=1}^{m-1}|\varphi_{i}(h)| \le |h| = |w|$.
  Therefore we conclude that 
  \[
  \sum \limits _{k =0}^{m-1} \vert g_k \vert \leq \vert g \vert -2 < \vert g \vert. \qedhere
  \] 
\end{proof}

\begin{lemma}
\label{lemma length red 3 for t equals 0}
 Let ${\mathfrak{v}} = 0 \,\overset{s-1}\dots \, 0$ and let $g = (g_0,\dots,g_{m-1})\sigma^{\epsilon} \in {\mathfrak{K(v)}}$ with $\epsilon\in\{\pm 1\}$.
 Suppose ${b}_1 \cdots {b}_\ell \in S^*$ is a geodesic word representing~$g$. If there exist $1 \leq r < r' \leq \ell$ such that either ${b}_r = a_{0}$, ${b}_{r'}=a_0^{-1}$ or ${b}_r = a_{0}^{-1}$, ${b}_{r'}=a_0$, then $\vert g_0g_1\cdots g_{m-1} \vert < \vert g \vert$ if $\epsilon=1$, and $\vert g_0g_{m-1}\cdots g_{1} \vert < \vert g \vert$ if $\epsilon=-1$. 
\end{lemma}

\begin{proof}
Note that we may assume that there are no $1\le i<j\le \ell$ such that $b_i=a_0^{-1}$ and $b_j=a_0$, as then  the desired length reduction is clear. It then follows that $$
w:={b}_1 \cdots {b}_\ell =*a_0(*a_0)\overset{\lambda}\cdots(*a_0)(*a_0^{-1})\overset{\lambda}\cdots (*a_0^{-1})*\qquad\text{if }\epsilon=1
$$ 
and
$$
w:={b}_1 \cdots {b}_\ell =(*a_0)\overset{\lambda}\cdots(*a_0)(*a_0^{-1})\overset{\lambda}\cdots (*a_0^{-1})*a_0^{-1}*\qquad\text{if }\epsilon=-1,
$$ 
for some $\lambda\in\N$ and  the $*\in\langle a_1,\ldots,a_{s-1}\rangle$ are arbitrary. 

Let \[
w'=(*a_0)\overset{\lambda}\cdots(*a_0)(*a_0^{-1})\overset{\lambda}\cdots (*a_0^{-1})*.
\]
Write $w'=(w_0',\ldots,w_{m-1}')$ and 
$\lambda=\mu m+\chi$ for some $\mu\in \mathbb{N}_0$ and $\chi\in [0,m-1]$. If $\mu\ne 0$, the first-level sections of $w'$ have the following form  
    \begin{align*}
   w_0' &=\varphi_0(*)(a_{s-1}\varphi_0(*))\overset{\mu-1}{\cdots}(a_{s-1}\varphi_0(*))a_{s-1}\varphi_0(*)a_{s-1}^{-1} (\varphi_0(*)a_{s-1}^{-1})\overset{\mu-1}{\cdots}(\varphi_0(*)a_{s-1}^{-1}) \varphi_0(*)\\
   w_1'&=(a_{s-1}\varphi_0(*))\overset{\mu-1}{\cdots}(a_{s-1}\varphi_0(*))a_{s-1}\varphi_0(*)a_{s-1}^{-1}  (\varphi_0(*)a_{s-1}^{-1})\overset{\mu-1}{\cdots}(\varphi_0(*)a_{s-1}^{-1})\\
   &\,\,\,\vdots\\
   w_{m-1-\chi}'&=(a_{s-1}\varphi_0(*))\overset{\mu-1}{\cdots}(a_{s-1}\varphi_0(*))a_{s-1}\varphi_0(*)a_{s-1}^{-1}  (\varphi_0(*)a_{s-1}^{-1})\overset{\mu-1}{\cdots}(\varphi_0(*)a_{s-1}^{-1})\\
   w_{m-\chi}'&=(a_{s-1}\varphi_0(*))\overset{\mu}{\cdots}(a_{s-1}\varphi_0(*))a_{s-1}\varphi_0(*)a_{s-1}^{-1}  (\varphi_0(*)a_{s-1}^{-1})\overset{\mu}{\cdots}(\varphi_0(*)a_{s-1}^{-1})\\
    &\,\,\,\vdots\\
    w_{m-1}'&=(a_{s-1}\varphi_0(*))\overset{\mu}{\cdots}(a_{s-1}\varphi_0(*))a_{s-1}\varphi_0(*)a_{s-1}^{-1}  (\varphi_0(*)a_{s-1}^{-1})\overset{\mu}{\cdots}(\varphi_0(*)a_{s-1}^{-1}).
    \end{align*}
 When $\mu = 0$ and hence $\chi\ne 0$, 
 we get
    \begin{align*}
        w_0' &= \varphi_0(*)\\
        w_1' &= 1\\
        &\,\,\,\vdots\\
        w_{m-1-\chi}' &= 1 \\
        w_{m-\chi}' &= a_{s-1}\varphi_0(*)a_{s-1}^{-1}\\
        &\,\,\,\vdots\\
        w_{m-1}' &= a_{s-1}\varphi_0(*)a_{s-1}^{-1}.
    \end{align*}
    Hence a straightforward computation shows that $|g_0g_1\cdots g_{m-1}|<|g|$ if $\epsilon=1$, and that $\vert g_0g_{m-1}\cdots g_{1} \vert < \vert g \vert$ if $\epsilon=-1$, as required.
\end{proof}

 
\section{Maximal subgroups} 
\label{sec: maximal subgroups}

 It follows from \cref{Proposition}{prop vir nil} below together with \cite[Proposition 2.21]{Fra20} that the group $\mathfrak{K(v)}$ admits maximal subgroups of infinite index if and only if it admits a proper subgroup $H < \mathfrak{K(v)}$ such that $HN = \mathfrak{K(v)}$ for every non-trivial normal subgroup $N \trianglelefteq \mathfrak{K(v)}$. A subgroup $H \leq \mathfrak{K(v)}$ satisfying the above condition is called a \emph{prodense subgroup}. 
  As seen below, we prove that $\mathfrak{K(v)}$, for a constant word~$\mathfrak{v}$, does not admit any proper prodense subgroup, which proves \cref{Theorem}{thm main}. 
  
 \begin{prop}\label{prop vir nil}
  The group $\mathfrak{K(v)}$ is just non-(virtually nilpotent). Hence, maximal subgroups of proper quotients of $\mathfrak{K(v)}$ are of finite index.
 \end{prop} 

 \begin{proof}
 As $\mathfrak{K(v)}$ has exponential word growth from \cref{Lemma}{lemma_k(v)_exp}, it follows from  
 Bass~\cite{Bass} and Guivarc'h~\cite{Guivarch}
 that  $\mathfrak{K(v)}$ is not virtually nilpotent. To see that every proper quotient of $\mathfrak{K(v)}$ is virtually nilpotent, by 
 \cite[Theorem 4.10]{Fra20}, it suffices to prove that $\mathfrak{K(v)}/\mathfrak{K(v)}''$ is virtually nilpotent. Set $N = \psi^{-1}(\gamma_3(\mathfrak{K(v)}) \times \cdots \times \gamma_3(\mathfrak{K(v)}))$. 
 From 
\cref{Lemma}{lemma_direct_product_gamma_3}, we have $N \leq \mathfrak{K(v)}'' < \St_{\mathfrak{K(v)}}(1) < {\mathfrak{K(v)}}$.
 Therefore ${\psi}$ induces a homomorphism 
   \[
    {\widetilde{\psi}}: \St_{\mathfrak{K(v)}}(1)/N \longrightarrow \mathfrak{K(v)}/\gamma_3(\mathfrak{K(v)}) \times \overset{m}\cdots \times \mathfrak{K(v)}/\gamma_3(\mathfrak{K(v)}).
  \]
 Since $\widetilde{\psi}$ is injective and $\widetilde{\psi}(\St_{\mathfrak{K(v)}}(1)/N)$ is nilpotent (being a subgroup of a nilpotent group), we obtain that $\St_{\mathfrak{K(v)}}(1)/N$ is nilpotent. This implies that $\St_{\mathfrak{K(v)}}(1)/{\mathfrak{K(v)}}''$ is nilpotent as it is a quotient of $\St_{\mathfrak{K(v)}}(1)/N$. As the subgroup $\St_{\mathfrak{K(v)}}(1)$ has finite index in $\mathfrak{K(v)}$, the group $\St_{\mathfrak{K(v)}}(1)/\mathfrak{K(v)}''$ has finite index in $\mathfrak{K(v)}/\mathfrak{K(v)}''$ and hence $\mathfrak{K(v)}/\mathfrak{K(v)}''$ is virtually nilpotent. The last part of the result follows from \cite[Corollary 5.1.3]{Fra19}.
\end{proof}
 
 Hereafter, for $g,h \in \mathfrak{K(v)}$, the equivalence $g \equiv h \mod \mathfrak{K(v)}'$ will simply be denoted by $g \equiv h$.  Notice that for every $z \in \mathfrak{K(v)}'$, we have $\sigma_z =1$ and $\mathfrak{K(v)}' \leq \St_{\mathfrak{K(v)}}(1)$.  Recall also from  \cref{Subsection}{subsec:2.1} the map $\varphi_u$ for $u\in X^*$. For  convenience, we introduce the following notation. For $i \in [0,s-2]$ and $j \in [0,m-1]$, we define
\begin{align*}
    \delta_{j}(a_i) = \begin{cases}
                     a_i &\text{ if } x_i = j,\\
                     1   &\text{otherwise},
\end{cases}
\end{align*}
 and extend $\delta_j$ to a homomorphism on $S^*$.
 \begin{lemma}
 \label{lemma g equiv product}
 Let ${\mathfrak{v}} = x_0 \cdots x_{s-2}$ be a word in the alphabet $X$, and let $g \in \mathfrak{K(v)}$ be such that $g \equiv a_{s-1}^{\epsilon_{s-1}} \cdots a_{0}^{\epsilon_{0}}$,
 where $\epsilon_{i}\in\{\pm 1\}$. For $j \in \mathbb{N}_0$, write $j = \ell s +r$, where $\ell \in \mathbb{N}_0$ and $r \in [0, s-1]$. If $\psi_j(g^{m^j})=(g_0,\dots,g_{m^j-1})$ then $g_k \equiv a_{s-1-r}^{\epsilon_{s-1}} \cdots a_{0-r}^{\epsilon_{0}}
  $ for all $k \in \intseg{0}{m^j-1}$.
 \end{lemma}
 
 \begin{proof}
 Since $g \equiv a_{s-1}^{\epsilon_{s-1}} \cdots a_{0}^{\epsilon_{0}}$, there exists an element $(z_0,\dots,z_{m-1})= z \in \mathfrak{K(v)}'$ such that $g = a_{s-1}^{\epsilon_{s-1}} \cdots a_{0}^{\epsilon_{0}} z$. Suppose that $\epsilon_0 = 1$. We have   that the section of $g$ at the vertex~$j$ is
 \begin{align*}
    (a_{s-1}^{\epsilon_{s-1}} \cdots a_{0}^{\epsilon_{0}} z)_j= \begin{cases}
        \delta_j(a_{s-2}^{\epsilon_{s-1}}\cdots a_0^{\epsilon_1} )z_{j+1} & \text{ if } j \in [0,m-2],\\
        \delta_{m-1}(a_{s-2}^{\epsilon_{s-1}}\cdots a_0^{\epsilon_1} ) a_{s-1}z_0 & \text{ if } j = m-1.
    \end{cases}
 \end{align*}
 Therefore, for $j \in [0,m-1]$, the element $\varphi_j(g^m)$ is equal to the product 
  \begin{align*}
      &\big(\delta_j(a_{s-2}^{\epsilon_{s-1}} \cdots a_0^{\epsilon_1} )z_{j+1}\big)\cdot\,\cdots\,\cdot\big( \delta_{m-2}( a_{s-2}^{\epsilon_{s-1}} \cdots a_0^{\epsilon_1})z_{m-1} \big)\cdot\\
      &\qquad\qquad\big(\delta_{m-1}(a_{s-2}^{\epsilon_{s-1}}\cdots a_0^{\epsilon_1}) a_{s-1}z_0\big)\cdot\\
        &\big(\delta_0(a_{s-2}^{\epsilon_{s-1}}\cdots a_0^{\epsilon_1})z_{1}\big)\cdot\,\cdots \,\cdot\big( \delta_{j-1}(a_{s-2}^{\epsilon_{s-1}}\cdots a_0^{\epsilon_1})z_{j}\big).
 \end{align*}
 Since $\delta_k(a_i^{\epsilon_{i+1}})$ is non-trivial for only one $k \in [0,m-1]$, we get that $\varphi_j(g^m) \equiv a_{s-2}^{\epsilon_{s-1}} \cdots a_{0}^{\epsilon_{1}}a_{s-1}^{\epsilon_{0}}$ for all $j \in \intseg{0}{m-1}$, since $z_0 \cdots z_{m-1} \in \mathfrak{K(v)}'$ by \cref{Lemma}{lemma commutator product}. The case $\epsilon_0 = -1$ follows in a similar way. The result then follows upon repeating the above process.
 \end{proof}

In the following, we denote by $\Sym(s)$ the symmetric group on  $[0,s-1]$.
 
 \begin{lemma}\label{lem:right-projection} 
 Let ${\mathfrak{v}}=t\,\overset{s-1}\dots \,t$ be a constant word, where $t\in[0,m-1]$. Let $g=a_{(s-1)^{\pi}}^{\epsilon_{(s-1)^{\pi}}}\cdots a_{0^{\pi}}^{\epsilon_{0^{\pi}}} \in \mathfrak{K(v)}$ where $\pi \in \Sym(s)$ and $\epsilon_i \in \{\pm 1\}$. For $j \in \mathbb{N}_0$, write  $j = \ell s +r$, where $\ell \in \mathbb{N}_0$ and $r \in [0, s-1]$. Then
 $\varphi_{t^j}(g^{m^j})=a_{(s-1)^{\pi}-r}^{\epsilon_{(s-1)^{\pi}}}\cdots a_{0^{\pi}-r}^{\epsilon_{0^{\pi}}}$. 
 \end{lemma} 
 
 \begin{proof}
  Let $i \in \intseg{0}{s-1}$ be such that $0 = i ^{\pi}$. Then
  \begin{align*}
     g&= a_{(s-1)^{\pi}}^{\epsilon_{(s-1)^{\pi}}}\cdots a_{(i+1)^{\pi}}^{\epsilon_{(i+1)^{\pi}}}a_0^{\epsilon_{0}}a_{(i-1)^{\pi}}^{\epsilon_{(i-1)^{\pi}}}\cdots a_{0^{\pi}}^{\epsilon_{0^{\pi}}}\\
     &={\begin{cases}(1,\overset{t-1}\dots,1, a_{(i-1)^{\pi}-1}^{\epsilon_{(i-1)^{\pi}}}\cdots a_{0^{\pi}-1}^{\epsilon_{0^{\pi}}},a_{(s-1)^{\pi}-1}^{\epsilon_{(s-1)^{\pi}}}\cdots a_{(i+1)^{\pi}-1}^{\epsilon_{(i+1)^{\pi}}}, 1,\ldots,1,a_{s-1})\sigma & \text{ if } \epsilon_0 =1\text{ and  }t>0,\\
     (a_{(s-1)^{\pi}-1}^{\epsilon_{(s-1)^{\pi}}}\cdots a_{(i+1)^{\pi}-1}^{\epsilon_{(i+1)^{\pi}}}, 1,\ldots,1,a_{s-1}a_{(i-1)^{\pi}-1}^{\epsilon_{(i-1)^{\pi}}}\cdots a_{0^{\pi}-1}^{\epsilon_{0^{\pi}}})\sigma & \text{ if } \epsilon_0 =1\text{ and  }t=0,\\
(a_{s-1}^{-1},1,\overset{t-1}\dots,1,a_{(s-1)^{\pi}-1}^{\epsilon_{(s-1)^{\pi}}}\cdots a_{(i+1)^{\pi}-1}^{\epsilon_{(i+1)^{\pi}}},  a_{(i-1)^{\pi}-1}^{\epsilon_{(i-1)^{\pi}}}\cdots a_{0^{\pi}-1}^{\epsilon_{0^{\pi}}},1,\ldots,1)\sigma^{-1} & \text{ if } \epsilon_0 =-1\text{ and  }t>0,\\
(a_{(s-1)^{\pi}-1}^{\epsilon_{(s-1)^{\pi}}}\cdots a_{(i+1)^{\pi}-1}^{\epsilon_{(i+1)^{\pi}}}a_{s-1}^{-1},  a_{(i-1)^{\pi}-1}^{\epsilon_{(i-1)^{\pi}}}\cdots a_{0^{\pi}-1}^{\epsilon_{0^{\pi}}},1,\ldots,1)\sigma^{-1} & \text{ if } \epsilon_0 =-1\text{ and  }t=0.
     \end{cases}
     }
   \end{align*}
  By taking the $m$th power of the element $g$ we get
\begin{align*}
     &g^m=\\
     &{\begin{cases}(*,\overset{t}\dots,*,a_{(s-1)^{\pi}-1}^{\epsilon_{(s-1)^{\pi}}}\cdots a_{(i+1)^{\pi}-1}^{\epsilon_{(i+1)^{\pi}}} a_{s-1}a_{(i-1)^{\pi}-1}^{\epsilon_{(i-1)^{\pi}}}\cdots a_{0^{\pi}-1}^{\epsilon_{0^{\pi}}},\#,\ldots,\#) & \text{ if } \epsilon_0 =1\text{ and  }t>0,\\
    (a_{(s-1)^{\pi}-1}^{\epsilon_{(s-1)^{\pi}}}\cdots a_{(i+1)^{\pi}-1}^{\epsilon_{(i+1)^{\pi}}} a_{s-1}a_{(i-1)^{\pi}-1}^{\epsilon_{(i-1)^{\pi}}}\cdots a_{0^{\pi}-1}^{\epsilon_{0^{\pi}}},\#,\ldots,\#) & \text{ if } \epsilon_0 =1\text{ and  }t=0,\\
(* ,\overset{t}\dots,*,a_{(s-1)^{\pi}-1}^{\epsilon_{(s-1)^{\pi}}}\cdots a_{(i+1)^{\pi}-1}^{\epsilon_{(i+1)^{\pi}}}a_{s-1}^{-1}  a_{(i-1)^{\pi}-1}^{\epsilon_{(i-1)^{\pi}}}\cdots a_{0^{\pi}-1}^{\epsilon_{0^{\pi}}},\#,\ldots,\#) & \text{ if } \epsilon_0 =-1\text{ and  }t>0,\\
(a_{(s-1)^{\pi}-1}^{\epsilon_{(s-1)^{\pi}}}\cdots a_{(i+1)^{\pi}-1}^{\epsilon_{(i+1)^{\pi}}}a_{s-1}^{-1}  a_{(i-1)^{\pi}-1}^{\epsilon_{(i-1)^{\pi}}}\cdots a_{0^{\pi}-1}^{\epsilon_{0^{\pi}}},\#,\ldots,\#)\sigma^{-1} & \text{ if } \epsilon_0 =-1\text{ and  }t=0.
     \end{cases}
     }
  \end{align*}
  where 
  \begin{align*}
 * = \begin{cases}a_{(i-1)^{\pi}-1}^{\epsilon_{(i-1)^{\pi}}}\cdots a_{0^{\pi}-1}^{\epsilon_{0^{\pi}}}a_{(s-1)^{\pi}-1}^{\epsilon_{(s-1)^{\pi}}}\cdots a_{(i+1)^{\pi}-1}^{\epsilon_{(i+1)^{\pi}}}a_{s-1} & \text{ if } \epsilon_0 =1, \\
 \vspace{0.001 in} &\\
 a_{s-1}^{-1}a_{(i-1)^{\pi}-1}^{\epsilon_{(i-1)^{\pi}}}\cdots a_{0^{\pi}-1}^{\epsilon_{0^{\pi}}}a_{(s-1)^{\pi}-1}^{\epsilon_{(s-1)^{\pi}}}\cdots a_{(i+1)^{\pi}-1}^{\epsilon_{(i+1)^{\pi}}} &  \text{ if } \epsilon_0 = -1,
 \end{cases}
 \end{align*}
 and
 \begin{align*}
 \# = \begin{cases}a_{s-1}a_{(i-1)^{\pi}-1}^{\epsilon_{(i-1)^{\pi}}}\cdots a_{0^{\pi}-1}^{\epsilon_{0^{\pi}}}a_{(s-1)^{\pi}-1}^{\epsilon_{(s-1)^{\pi}}}\cdots a_{(i+1)^{\pi}-1}^{\epsilon_{(i+1)^{\pi}}} & \text{ if } \epsilon_0 =1, \\
 \vspace{0.001 in} &\\
 a_{(i-1)^{\pi}-1}^{\epsilon_{(i-1)^{\pi}}}\cdots a_{0^{\pi}-1}^{\epsilon_{0^{\pi}}}a_{(s-1)^{\pi}-1}^{\epsilon_{(s-1)^{\pi}}}\cdots a_{(i+1)^{\pi}-1}^{\epsilon_{(i+1)^{\pi}}}a_{s-1}^{-1} &  \text{ if } \epsilon_0 = -1.
 \end{cases}
 \end{align*}
 
 In particular, we have  $\varphi_{t}(g^{m})=a_{(s-1)^{\pi}-1}^{\epsilon_{(s-1)^{\pi}}}\cdots a_{0^{\pi}-1}^{\epsilon_{0^{\pi}}}$, and the result follows recursively.
 \end{proof}

 We recall that $H_u$ denotes the subgroup $\varphi_u(\st_H(u))$ for a vertex $u \in X^*$. By \cref{Lemma}{lemma_k(v)_fractal_and_level-trans}, 
 the group~$\mathfrak{K(v)}$ is fractal, so $
 \mathfrak{K(v)}_u=
 \mathfrak{K(v)}$ for all $u \in X^*$. 
   
 \begin{lemma}\label{lemma cyclic conj} 
Let ${\mathfrak{v}}=t\,\overset{s-1}\dots\, t$ be a constant word, where $t\in[0,m-1]$, and let $H$ be a subgroup of~$\mathfrak{K(v)}$. Assume that $a_{(s-1)^{\pi}}^{\epsilon_{(s-1)^{\pi}}}\cdots a_{0^{\pi}}^{\epsilon_{0^{\pi}}} \in H$ for some $\pi \in \Sym(s)$,  where $\epsilon_i \in \{\pm 1\}$.
   Then the following assertions hold.
   \begin{enumerate}
       \item[(i)] For each $n \in \mathbb{N}$ and vertex~$u$ of level~$ns$, the subgroup~$H_u$ contains a cyclic 
       permutation of the word $a_{(s-1)^{\pi}}^{\epsilon_{(s-1)^{\pi}}}\cdots a_{0^{\pi}}^{\epsilon_{0^{\pi}}}$.
       \item[(ii)] Furthermore, 
       if $\epsilon_i=1$ for some $i\in[0,s]$, then for each $n \in \mathbb{N}$, there is a vertex~$u$ of level~$ns$ such that the cyclic permutation~$w$ of $a_{(s-1)^{\pi}}^{\epsilon_{(s-1)^{\pi}}}\cdots a_{0^{\pi}}^{\epsilon_{0^{\pi}}}$ that is contained in~$H_u$ satisfies the following property: \begin{enumerate}
           \item [(a)] if $t>0$, the cyclic permutation $w$ ends with $a_i$ on the right;
           \item [(b)] if $t=0$, the cyclic permutation $w$ starts with $a_i$ on the left.
       \end{enumerate} 
  \end{enumerate}
 \end{lemma}
 
 \begin{proof}
 (i) Let $g = a_{(s-1)^{\pi}}^{\epsilon_{(s-1)^{\pi}}}\cdots a_{0^{\pi}}^{\epsilon_{0^{\pi}}} \in H$. Then $0 = i ^{\pi}$ for some $i \in \intseg{0}{s-1}$.  Observe from the proof of \cref{Lemma}{lem:right-projection} that 
 \[\varphi_{t}(g^m)=a_{(s-1)^{\pi}-1}^{\epsilon_{(s-1)^{\pi}}}\cdots a_{(i+1)^{\pi}-1}^{\epsilon_{(i+1)^{\pi}}}a_{s-1}^{\epsilon_{0}}a_{(i-1)^{\pi}-1}^{\epsilon_{(i-1)^{\pi}}}\cdots a_{0^{\pi}-1}^{\epsilon_{0^{\pi}}}\]
 and that $\varphi_j(g^m)$ is a cyclic permutation of $\varphi_{t}(g^m)$ for every $j \in \intseg{0}{m-1}$. By repeating the process of taking powers we get that ${\psi}_s(g^{m^s})=(g_0,\dots,g_{m^s-1})$ with $g_{\tau}= g$, for the component~$\tau$ corresponding to the vertex $t^s$, and $g_k$ is a cyclic permutation of the word $a_{(s-1)^{\pi}}^{\epsilon_{(s-1)^{\pi}}}\cdots a_{0^{\pi}}^{\epsilon_{0^{\pi}}}$ for $k\in [ 0,m^s-1]$. 
 
\noindent (ii)  We will only consider the case $t>0$, as the case $t=0$ is analogous. In particular,  if $\epsilon_0 =1$, we note from the proof of \cref{Lemma}{lem:right-projection} that
 \[
 \varphi_{t-1}(g^m)=a_{(i-1)^{\pi}-1}^{\epsilon_{(i-1)^{\pi}}}\cdots a_{0^{\pi}-1}^{\epsilon_{0^{\pi}}}a_{(s-1)^{\pi}-1}^{\epsilon_{(s-1)^{\pi}}}\cdots a_{(i+1)^{\pi}-1}^{\epsilon_{(i+1)^{\pi}}}a_{s-1},
 \]
 and by \cref{Lemma}{lem:right-projection} we see that
 \begin{align*}
 \varphi_{(t-1)t^{s-1}}(g^m) = 
 a_{(i-1)^{\pi}}^{\epsilon_{(i-1)^{\pi}}}\cdots a_{0^{\pi}}^{\epsilon_{0^{\pi}}}a_{(s-1)^{\pi}}^{\epsilon_{(s-1)^{\pi}}}\cdots a_{(i+1)^{\pi}}^{\epsilon_{(i+1)^{\pi}}}a_{0}.
 \end{align*}
 More generally, suppose  that $\epsilon_j =1$ for $j\in[0,s]$ and let $\ell\in[0,s]$ be such that $\ell^\pi=j$. Writing $v_j= t\overset{j}\cdots t$ and $w_j=(t-1)t\,\overset{s-j-1}\cdots\, t$, we recall from \cref{Lemma}{lem:right-projection} that
 \[
 \varphi_{v_j}(g^{m^j})=a_{(s-1)^{\pi}-j}^{\epsilon_{(s-1)^{\pi}}}\cdots a_{(\ell+1)^{\pi}-j}^{\epsilon_{(\ell+1)^{\pi}}}a_{0}a_{(\ell-1)^{\pi}-j}^{\epsilon_{(\ell-1)^{\pi}}}\cdots a_{0^{\pi}-j}^{\epsilon_{0^{\pi}}}.
 \]
 Then  similar to the above  we see that
 \begin{align*}
 \varphi_{v_jw_j}(g^m) = 
 a_{(\ell-1)^{\pi}}^{\epsilon_{(\ell-1)^{\pi}}}\cdots a_{0^{\pi}}^{\epsilon_{0^{\pi}}}a_{(s-1)^{\pi}}^{\epsilon_{(s-1)^{\pi}}}\cdots a_{(\ell+1)^{\pi}}^{\epsilon_{(\ell+1)^{\pi}}}a_{j},
 \end{align*}
 and as $u_j:=v_jw_j$ is a vertex of level $s$, we have that $H_{u_j}$ contains a cyclic permutation of $a_{(s-1)^{\pi}}^{\epsilon_{(s-1)^{\pi}}}\cdots a_{0^{\pi}}^{\epsilon_{0^{\pi}}}$  that  ends with $a_j$ on the right.

 Now, by using \cref{Lemma}{lem:right-projection} repeatedly,  one can see that the result holds for  level~$ns$ of~$T$, for $n>1$.
 \end{proof}

 \begin{prop} \label{prop g equiv product}
  Let ${\mathfrak{v}}=t\,\overset{s-1}\dots\, t$ be a constant word, where $t\in[0,m-1]$, and let $ g \in \mathfrak{K(v)}$ be such that $g \equiv a_{s-1}^{\epsilon_{s-1}}\cdots a_{0}^{\epsilon_{0}}$, where $\epsilon_{i}\in\{\pm 1\}$. Then there exists a vertex $u$ of  level $ns$ in~$T$,  for some $n \in \mathbb{N}_0$, and an element $g' \in \textup{st}_{\langle g \rangle}(u)$ such that $\varphi_u(g')=a_{(s-1)^{\pi}}^{\epsilon_{(s-1)^{\pi}}}\cdots a_{0^{\pi}}^{\epsilon_{0^{\pi}}}$ for some $\pi \in \Sym(s)$. 
 \end{prop}

\begin{proof}
 The proof proceeds by induction on the length of $g$. Recall from 
 \cref{Lemma}{lemma_abelian}
 that $\mathfrak{K(v)}/\mathfrak{K(v)}'=\langle a_0\mathfrak{K(v)}',\ldots,a_{s-1}\mathfrak{K(v)}'\rangle\cong \mathbb{Z}^s$. Hence if $g$ is equivalent to $a_{s-1}^{\epsilon_{s-1}}\cdots a_{0}^{\epsilon_{0}}$ then $|g| > s-1$, since any word containing each of the distinct generators of $\mathfrak{K(v)}$ has length at least~$s$. Assume that $\vert g \vert = s$. Then 
  \[
   g \in \big\{a_{(s-1)^{\pi}}^{\epsilon_{(s-1)^{\pi}}}\cdots  a_{0^{\pi}}^{\epsilon_{0^{\pi}}} \mid \pi \in \Sym(s)\big\}
  \]
 and the result follows trivially by choosing $u$ as the root vertex.
 Now, assume that $\vert g \vert > s$. Since the exponent sum of~$a_0$ in any word representing $g$ is $\epsilon_0$, we can write $g = (g_0,\dots, g_{m-1})\sigma^{\epsilon_0}$ with $g_0, \dots,g_{m-1}\in \mathfrak{K(v)}$. We get
   \[
    g^m = \begin{cases}(g_0g_1\cdots g_{m-1},\,g_1\cdots g_{m-1}g_0\,,\,\dots\,,\, g_{m-1}g_0g_1\cdots g_{m-2}) & \text{ if } \epsilon_0 =1,\\
    (g_0g_{m-1}g_{m-2}\cdots g_1,\,g_1g_0g_{m-1}
    \cdots g_2\,,\,\dots\,,\, g_{m-1}g_{m-2}\cdots g_{0}) & \text{ if } \epsilon_0 = -1.
    \end{cases}
   \]
 For every $k \in \intseg{0}{m-1}$, we set $\alpha_k = \varphi_{k}(g^m)$.
 It follows from \cref{Lemma}{lemma g equiv product} that 
 \[
 \alpha_k \equiv a_{s-2}^{\epsilon_{s-1}}\cdots  a_0^{\epsilon_1} a_{s-1}^{\epsilon_0}
 \]
 for all $k \in \intseg{0}{m-1}$. Furthermore $\vert \alpha_k \vert \leq \vert g \vert$ for all $k \in \intseg{0}{m-1}$ by \cref{Lemma}{lemma length red 2}. If there exists $k \in \intseg{0}{m-1}$ such that $\vert \alpha_k \vert < \vert g \vert$, then  it follows by induction that there exists a vertex $u$ of level $ns$ in $T$, for some $n \in \mathbb{N}_0$, and $g' \in \st_{\langle \alpha_k \rangle}(u)$ such that
 \[
 \varphi_u(g')=a_{(s-1)^{\pi}-1}^{\epsilon_{(s-1)^{\pi}}}a_{(s-2)^{\pi}-1}^{\epsilon_{(s-2)^{\pi}}} \cdots a_{0^{\pi}-1}^{\epsilon_{0^{\pi}}}
 \]
for some $\pi \in \Sym(s)$.  Using 
 \cref{Lemma}{lem:right-projection}, 
 we get that
 \[\varphi_{t^{s-1} }((g')^{m^{s-1}}) = a_{(s-1)^{\pi}}^{\epsilon_{(s-1)^{\pi}}}a_{(s-2)^{\pi}}^{\epsilon_{(s-2)^{\pi}}} \cdots a_{0^{\pi}}^{\epsilon_{0^{\pi}}},\]
 and hence the result follows.
 
 Assume that $\vert \alpha_k \vert = \vert g \vert$ for all $k \in \intseg{0}{m-1}.$ Since $\vert \alpha_k \vert \leq \sum \limits _{\ell =0}^{m-1} \vert g_{\ell} \vert$, in particular, we get
 \[
   \sum \limits _{\ell =0}^{m-1} \vert g_{\ell} \vert = \vert g \vert.
 \]
 Let $w_g \in S^*$ be a geodesic word representing~$g$. Since for each $i\in[0,s-1]$ the element~$a_i^{\epsilon_i}$ contributes $a_{i-1}^{\epsilon_i}$ in exactly one component, we can obtain 
 words $w_{g_k}$ representing~$g_k$ 
 by substituting~$a_i^{\epsilon_i}$ in~$w_g$ with~$a_{i-1}^{\epsilon_i}$ in the appropriate component.  Notice that $\vert g_k\vert \leq \vert w_{g_k} \vert$ for every $k \in \intseg{0}{m-1}$. Moreover, the words $w_{g_k}$ are geodesic. Indeed, 
 \[
  \sum \limits_{k = 0}^{m-1}\vert g_k\vert \leq \sum \limits_{k = 0}^{m-1}\vert w_{g_k}\vert \leq \vert w_g  \vert = \vert g \vert =  \sum \limits _{\ell =0}^{m-1} \vert g_{\ell} \vert,
 \]
 which forces that $\vert w_{g_k} \vert = \vert g_k \vert.$ 
 Now,
 set 
 \[
  w_{\alpha_k} = \begin{cases} 
                  w_{g_k}w_{g_{k+1}}\cdots w_{g_{k+m-1}} & \text{ if } \epsilon_0 = 1,\\
                  w_{g_k}w_{g_{k-1}}\cdots w_{g_{k-(m-1)}} & \text{ if } \epsilon_0 = -1.
  \end{cases}
 \]
 Clearly $w_{\alpha_k}$ represents $\alpha_k$. Therefore $\vert \alpha_k \vert \leq \vert w_{\alpha_k} \vert$. Furthermore, 
   \[
    \vert w_{\alpha_k} \vert \leq \sum \limits_{\ell=0}^{m-1}\vert w_{g_{\ell}} \vert= \sum \limits _{\ell =0}^{m-1} \vert g_{\ell} \vert = \vert g \vert = \vert \alpha_k \vert .
   \] 
  Thus $\vert \alpha_k \vert = \vert w_{\alpha_k} \vert$ and $w_{\alpha_k}$ is a geodesic word. 
   
   Now, we claim that in order to prove the result, it suffices to consider the situation in which for every $i \in \intseg{0}{s-1}$ there exists a unique $k \in \intseg{0}{m-1}$ such that $w_{g_k}$ contains a non-trivial power of $a_i$. First we consider the case when $i=0$. Assume to the contrary that there exist 
  distinct $k_1,k_2 \in \intseg{0}{m-1}$  such that $w_{g_{k_1}}$ and $w_{g_{k_2}}$
   contain non-trivial powers of~$a_0$. We can reduce to the following two cases.

 \smallskip
 
 \underline{Case 1:}
 Suppose that there exist 
distinct $k_1,k_2 \in \intseg{0}{m-1}$ such that $w_{g_{k_1}}$ and $w_{g_{k_2}}$ contain~$a_0$ and $a_0^{-1}$ respectively.
 Then for some $k\in \intseg{0}{m-1}$, the word~$w_{\alpha_k}$ contains a subword of the form $a_0 w a_0^{-1}$ with $w \in S^*$. If $\alpha_k^m = (\beta_0,\dots,\beta_{m-1})$ then, 
 by \cref{Lemma}{lemma length red 2}, \cref{Lemma}{lemma length red 3} and \cref{Lemma}{lemma length red 3 for t equals 0},
 we obtain that $\vert \beta_\ell \vert < \vert \alpha \vert$ for 
 $\ell \in \intseg{0}{m-1}$. 
 Again, the result follows by induction.

  \smallskip
  
 \underline{Case 2:} Suppose there exist distinct
 $k_1,k_2 \in \intseg{0}{m-1}$ such that $w_{g_{k_1}}$ and $w_{g_{k_2}}$ contain $a_0$. Recall  that $\mathfrak{K(v)}/\mathfrak{K(v)}'=\langle a_0\mathfrak{K(v)}',\ldots,a_{s-1}\mathfrak{K(v)}'\rangle\cong \mathbb{Z}^s$. Hence, as
 the exponent sum of $a_{1}$ in any word representing $g$ is $\epsilon_1$, the exponent sum of~$a_0$ in $w_{\alpha_k}$ is equal to $\epsilon_1$ for all $k \in \intseg{0}{m-1}$. This implies that there exists
 $k_3 \in \intseg{0}{m-1}$ such that $w_{g_{k_3}}$
 contains~$a_{0}^{-1}$, and we are in the previous case. Analogously, the same argument works if both
 $w_{g_{k_1}}$ and $w_{g_{k_2}}$ 
 contain $a_0^{-1}$. 
 
  \smallskip
  
 We reduce to the case where there exists a unique $k \in \intseg{0}{m-1}$ such that $w_{g_k}$ contains a non-trivial power of~$a_0$. By inducting on $i \in \intseg{0}{s-1}$, assume that there exists a unique $k \in \intseg{0}{m-1}$ such that $w_{g_k}$ contains a non-trivial power of $a_{i-1}$. Then suppose that there exist  distinct $k_1 ,k_2 \in \intseg{0}{m-1}$  such that $w_{g_{k_1}}$ and $w_{g_{k_2}}$ contain non-trivial powers of~$a_{i}$. We can find 
 $k_3 \in \intseg{0}{m-1}$ such that $w_{\alpha_{k_3}}$
 contains a subword of the form $a_{i}
 ^{\ell_1}wa_{i}
 ^{\ell_2}$, where $\ell_1,\ell_2 \in \mathbb{Z}\backslash \{0\}$ and $w \in S^*$ with exponent sum of~$a_0$ in~$w$ is not equal to $0 \bmod m$. Thanks to \cref{Lemma}{lemma g equiv product}, we may replace $g$ with $\alpha_{k_3}$. Then we find more than one $w_{g_k}$ containing non-trivial powers of~$a_{i-1}$, contradicting the induction hypothesis 
 and hence proving the claim. 
 
 Thus, we reduce to the situation in which for every $i \in \intseg{0}{s-1}$ there exists a unique $k \in \intseg{0}{m-1}$ such that $w_{g_k}$ contains a non-trivial power of~$a_i$. An easy computation yields that 
 $w_{g}$ does not contain a subword of the form $a_{i}^{\ell_1}wa_{i}^{\ell_2}$, for some $i \in \intseg{0}{s-1}$  where $\ell_1,\ell_2 \in \mathbb{Z}\backslash \{0\}$ and $w \in S^*$ with the exponent sum of~$a_0$ in~$w$ being not equal to $0 \bmod m$. Hence,  we conclude that $w_g$ must be of the form
  \[
    {w_1}(a_{i_1},\dots, a_{i_r})a_0^{\epsilon_0} {w_2}(a_{i_{r+1}},\dots, a_{i_{s-1}}),
  \]
 where {$w_1$ and $w_2$ are words in the given elements, and} $\{i_1,\dots,i_r,i_{r+1},\dots,i_{s-1}\}= \intseg{1}{s-1}$ such that the intersection $\{i_1,\dots,{i_r}
 \} \cap \{i_{r+1},\dots,i_{s-1}\}$ is empty.
 Consider the element $\alpha_{t}$ obtained from the element $g$ above. Then the corresponding $w_{\alpha_t}$ has the form 
\[
   w_1(a_{i_{1}-1},\dots, a_{i_{r}-1})a_{s-1}^{\epsilon_0} w_2(a_{i_{r+1}-1},\dots, a_{i_{s-1}-1}),
  \]
 and continuing the above procedure with this word, yields  the element $a_{(s-1)^{\pi}}^{\epsilon_{(s-1)^{\pi}}}\cdots a_{0^{\pi}}^{\epsilon_{0^{\pi}}} \in H_u$ for some $u$ of level $ns$ in $T$, for some $n \in \mathbb{N}_0$.
\end{proof}

 \begin{theorem}\label{prop prodense}
Let ${\mathfrak{v}}=t\,\overset{s-1}\dots \,t$ be a constant word, where $t\in[0,m-1]$.  If $H$ is a prodense subgroup of $\mathfrak{K(v)}$ then $H = \mathfrak{K(v)}$.
 \end{theorem}

 \begin{proof} 
  Note that $H\mathfrak{K(v)}' =\mathfrak{K(v)}$ as $H$ is a prodense subgroup. Therefore there exists an element 
  $z \in \mathfrak{K(v)}'$ such that $ a_{s-1}\cdots a_0z\in H$. By an application of \cref{Proposition}{prop g equiv product}, we can find $u\in T$ such that $H_{u}$ contains $a_{(s-1)^{\pi}}\cdots a_{0^{\pi}}$ for some $\pi \in \Sym(s)$. We set $g=a_{(s-1)^{\pi}}\cdots a_{0^{\pi}}$. Thanks to \cite[Lemma~3.1]{Fra20}, the subgroup~$H_u$ is again a prodense subgroup of~$\mathfrak{K(v)}$. Without loss of generality, we replace $H$ with~$H_u$.   
   
  Again, as $H$ is prodense, for some $\widetilde{z}\in \mathfrak{K(v)}'$ we similarly have  $a_{s-1}\cdots a_1a_0^{-1}\widetilde{z}\in H$. By \cref{Proposition}{prop g equiv product}, there exists a vertex $u$ at level $ns$, for some $n \in \mathbb{N}$, such that $H_u$ contains an element $h_0$ of the form 
  \[
    h_0 = a_{(s-1)^{\tau_0}}^{\,\epsilon_{(s-1)^{\tau_{0}}}}\cdots a_{0^{\tau_0}}^{\,\epsilon_{0^{\tau_0}}},
  \]
 where $\tau_0 \in \Sym(s)$, 
 with $\epsilon_{i^{\tau_0}}=-1$ if $i^{\tau_0}=0$ and $\epsilon_{i^{\tau_0}}=1$ otherwise. Now, by \cref{Lemma}{lemma cyclic conj}(i), the subgroup~$H_u$ also contains some cyclic permutation of the element~$g$. By abuse of notation, we replace~$g$ with this cyclic permutation of~$g$.  We again replace $H$ with $H_u$. Now $H$ contains the elements $g$ and $h_0$.
 Repeating this argument $s-1$ times, we may assume that
 $H$ contains the elements $g,h_0,\dots,h_{s-1}$, where
 \[
    h_j = a_{(s-1)^{\tau_j}}^{\,\epsilon_{(s-1)^{\tau_{j}}}}\cdots a_{0^{\tau_j}}^{\,\epsilon_{0^{\tau_j}}},
  \]
 where $\tau_j \in \Sym(s)$ 
 with $\epsilon_{i^{\tau_j}}=-1$ if $i^{\tau_j}=j$ and $\epsilon_{i^{\tau_j}}=1$ otherwise.

 For the remainder of the proof, we will assume that $t>0$; the case $t=0$ is analogous, as accounted for in \cref{Remark}{rmk:appendix}.
  Appealing to \cref{Lemma}{lemma cyclic conj}(ii), we now choose a vertex~$v$, with $v$ of level $\widetilde{n}s$ for some  $\widetilde{n} \in \mathbb{N}$, such that the cyclic permutation of~$g$ that is contained in $H_v$ ends with~$a_0$ on the right. We rename this element~$g$. 
  So we have $g\in H_v$ and by \cref{Lemma}{lemma cyclic conj}(i) we have a cyclic permutation of each of the elements $h_0,\ldots,h_{s-1}$ in~$H_v$. 
  By abuse of notation, we rename these cyclic permutations $h_0,\ldots,h_{s-1}$ respectively.
  As before we replace $H$ with~$H_v$. Now $H$ contains the elements $g,h_0,\dots,h_{s-1}$, where $g$ ends with~$a_0$ on the right. 
   
   For each $n \in \N$, let 
   $v_n=t \,\overset{n}\dots\, t$.  
   It follows from \cref{Lemma}{lem:right-projection} that for  $d\in\mathbb{N}$, $\varphi_{v_{ds}}(g^{m^{ds}})=g$ and $\varphi_{v_{ds}}(h_i^{\,m^{ds}})=h_i$ where $i\in[0,s-1]$. Furthermore, for any element  $f \in \mathfrak{K(v)}$ of the form
   \[
   f=a_{\iota_1}^{\epsilon_1} \cdots a_{\iota_b}^{\epsilon_b},
   \] 
   for pairwise distinct $\iota_1,\ldots,\iota_{b}\in[0,s-1]$ with ${b}\in[1,s]$ and $\epsilon_i\in\{\pm 1\}$,
   we can consider its contribution to $H_{v_n}$. Specifically, if $f \in \St_{\mathfrak{K(v)}}(1)$, we simply consider its image under $\varphi_{t}$. If $f \notin \St_{\mathfrak{K(v)}}(1)$, then we consider $\varphi_{t}(f^m)$. We refer to this general process as projecting along the path $t_{\infty}:=tt\cdots$.
   By projecting along  this path, we observe that if $f \in H_{v_j}$, for $j \in \mathbb{N}$, then $f \in H_{v_{j+s}}$; 
   compare the proof of \cref{Lemma}{lem:right-projection}. This observation will be used repeatedly throughout the proof without special mention.

    The strategy of the proof is now to consider the contributions  from 
    \[
    \langle g\rangle\,,\,\langle h_0\rangle\,,\,\ldots\,,\,\langle h_{s-1}\rangle
    \]
    to $H_{v_n}$, and to multiply them appropriately to separate the generators $a_0,\ldots,a_{s-1}$. More specifically, if  for some $n\in\mathbb{N}$, suppose we have non-trivial distinct elements 
    $\alpha,\beta\in H_{v_n}$ of the form
    \[
    \alpha=a_{i_1}^{\epsilon_1}\cdots a_{i_q}^{
    \epsilon_q},\qquad  \beta=a_{j_1}^{\delta_1}\cdots a_{j_r}^{\delta_r}
    \]
    where
    $\epsilon_i,\delta_j\in\{\pm 1\}$ and $2\le {q},r\le s$, with $i_1,\ldots,i_{q}\in[0,s-1]$  pairwise distinct, and also $j_1,\ldots,j_r\in[0,s-1]$  pairwise distinct. We consider two situations below, where we assume always that $\widetilde{\alpha},\widetilde{\beta}$ are non-trivial.
    \begin{enumerate}
        \item [(i)] If
        $\alpha= \widetilde{\alpha}a_0$ and $\beta= \widetilde{\beta}a_0^{-1}\widehat{\beta}$, then
        \[
        \beta\alpha=\widetilde{\beta}a_0^{-1}\widehat{\beta}\widetilde{\alpha}a_0
        \]
       yields
        \[
        \varphi_{t}(\widetilde{\beta})\in H_{v_{n+1}}, 
        \]
        and hence
        \[
        \widetilde{\beta}\in H_{v_{n+s}}\qquad\text{and}\qquad a_0^{-1}\widehat{\beta}\in H_{v_{n+s}}.
        \]

          \item [(ii)] If
       $       \alpha= a_0\widetilde{\alpha}$ and $\beta= \widehat{\beta}a_0^{-1}\widetilde{\beta}$, from 
        \[
        \alpha\beta=a_0\widetilde{\alpha}\widehat{\beta}a_0^{-1}\widetilde{\beta},
        \]
        we obtain
        \[
        \varphi_{t}(\widetilde{\beta})\in H_{v_{n+1}}, 
        \]
        and similarly,
        \[
        \widetilde{\beta}\in H_{v_{n+s}}\qquad\text{and}\qquad \widehat{\beta}a_0^{-1}\in H_{v_{n+s}}.
        \]
    \end{enumerate} 
    In other words, upon replacing $H_{v_n}$ with $H_{v_{n+s}}$ we have split $\beta\in H_{v_{n+s}}$ into two non-trivial parts. The plan is to repeatedly perform such operations as in (i) and (ii) above to keep splitting products of generators. Eventually we will end up with $a_0,\ldots,a_{s-1}\in H_u$ for some $u$, which gives $H_u=\mathfrak{K(v)}$ and equivalently that $H=\mathfrak{K(v)}$, as required. 
  
  \medskip

  We begin by first considering the contributions from $\langle g\rangle$ and $\langle h_0\rangle$ along the path $t_{\infty}$ of the tree.
For convenience, write
  \[
  g=a_{i_1}\cdots a_{i_{s-1}}a_0\qquad\text{and}\qquad  h_0=a_{j_1}\cdots a_{j_{d-1}}a_0^{-1}a_{j_{d+1}}\cdots a_{j_s},
  \]
  for some $d\in[1,s]$, where $\{i_1,\ldots,i_{s-1}\}=\{j_1,\ldots,j_{d-1},j_{d+1},\ldots,j_s\}=[1,s-1]$.

   \medskip
    
    \underline{Case 1:} Suppose $1<d<s$.
    Then we are in situation (i) from above, and   it follows that   
   \[
   a_{j_1-1}\cdots a_{j_{d-1}-1}\in H_{v_1}\quad\text{and}\quad a_{s-1}^{-1}a_{j_{d+1}-1}\cdots a_{j_s-1}\in H_{v_1}.
   \]
   We will now use (the projections of) these two parts of $h_0$ to split $g$ into two non-trivial parts.

   Let $j:=j_{d-1}$. We consider the contribution of $a_{j_1}\cdots a_{j_{d-1}}$ to $H_{v_j}$. 
   In other words, we project along the path $t_{\infty}$ down to level~$j$, which gives
   \[
   a_{j_1-j}\cdots a_{j_{d-2}-j
   }a_{0}\in H_{v_j}.
   \]
   Recalling that $g=a_{i_1}\cdots a_{i_{s-1}}a_0$, we have that $i_r=j$ for some $r\in[1,s-1]$. Then setting
   \[
  \beta^{-1}:= \varphi_{v_j}(g^{m^j})=a_{i_1-j}\cdots a_{i_{r-1}-j}a_0a_{i_{r+1}-j}\cdots a_{i_{s-1}-j}a_{s-j}\in H_{v_j}
   \]
   and 
  \[
  \alpha:= a_{j_1-j}\cdots a_{j_{d-2}-j
   }a_{0}\in H_{v_j},
  \]
  it follows from situation (i) that
  \[
   a_{i_{r+1}-j-1}\cdots a_{i_{s-1}-j-1}a_{s-j-1}\in H_{v_{j+1}}\quad\text{and}\quad   a_{i_1-j-1}\cdots a_{i_{r-1}-j-1}a_{s-1}\in H_{v_{j+1}},
   \]
   so we have split $g$ into two non-trivial parts. 
   
    We now use the two parts of $g$ to split the parts of ${h_0}$ further. For clarity, let us first project to $v_s$. Here in $H_{v_s}$ we have the elements
   \[
   a_{j_1}\cdots a_{j_{d-2}}a_j,\qquad a_0^{-1}a_{j_{d+1}}\cdots a_{j_s},\qquad a_{i_1}\cdots a_{i_{r-1}}a_j,\qquad a_{i_{r+1}}\cdots a_{i_{s-1}}a_0.
   \]
   The left two elements are the two parts of ${h_0}$, and the right two are those of $g$. Without loss of generality, we replace  $H$ with $H_{v_s}$.

   \textit{Subcase (a)}: Suppose $1<r<s-1$. Let $k:=i_1$. Then either $k=j_{q}$ for ${q}\in[1,d-2]$ or  $k=j_{q}$ for ${q}\in[d+1,s]$. Suppose the former; a similar argument works for the latter. If ${q}>1$, we let~$\beta^{-1}$ be the $k$th level projection of $a_{j_1}\cdots a_{j_{d-2}}a_j$ (as usual along the path $t_{\infty}$) and $\alpha$ be that of $a_ka_{i_2}\cdots a_{i_{r-1}}a_j$, which by (ii) gives,
  upon replacing $H$ with $H_{v_s}$, the following elements  in $H$:
    \[
   a_{j_1}\cdots a_{j_{q-1}},\quad a_ka_{j_{q+1}}\cdots a_{j_{d-2}}a_j,\quad a_0^{-1}a_{j_{d+1}}\cdots a_{j_s},\quad a_ka_{i_2}\cdots a_{i_{r-1}}a_j,\quad a_{i_{r+1}}\cdots a_{i_{s-1}}a_0.
   \]
   
 If ${q}=1$, before splitting, we have instead the following elements in~$H$:
   \[
   a_ka_{j_2}\cdots a_{j_{d-2}}a_j,\qquad a_0^{-1}a_{j_{d+1}}\cdots a_{j_s},\qquad a_ka_{i_2}\cdots a_{i_{r-1}}a_j,\qquad a_{i_{r+1}}\cdots a_{i_{s-1}}a_0.
   \]
   Hence we let ${\ell}:=i_{r+1}$ and let $c\in [2,d-2]\cup[d+1,s]$ be such that $j_c={\ell}$. We consider the ${\ell}$th projection  of ${a_{\ell}a_{i_{r+2}}\cdots a_{i_{s-1}}a_0}$ multiplied accordingly with that of ${a_j^{-1}a_{j_{d-2}}^{-1} \cdots a_{j_2}^{-1}
   a_k^{-1}}$ or ${a_{j_s}^{-1} \cdots a_{j_{d+1}}^{-1} 
   a_0}$. This is situation (ii). 
   
   \textit{Subcase (b)}: Suppose $r=1$. Then we have the following elements in $H$:
    \[
   a_{j_1}\cdots a_{j_{d-2}},\qquad a_0^{-1}a_{j_{d+1}}\cdots a_{j_s},\qquad a_j,\qquad a_{i_{2}}\cdots a_{i_{s-1}}a_0.
   \]
   If $i_2\ne j_1$, let $k:={j_1}$, and we  proceed according to (ii), {with $\alpha$ being the $k$th projection of $a_ka_{j_2}\cdots a_{j_{d-2}}$ and $\beta$ that of $( a_{i_{2}}\cdots a_{i_{s-1}}a_0)^{-1}$.}
   If $i_2= j_1$, we let instead $k:=j_s$ and consider the $k$th projection of $(a_{i_{2}}\cdots a_{i_{s-1}}a_0)^{-1}$ multiplied with that of $a_0^{-1}a_{j_{d+1}}\cdots a_{j_{s-1}}a_k$; that is, situation~(i).
   
   \textit{Subcase (c)}: Suppose $r=s-1$. Here we have the following elements in $H$:
   \[
   a_{j_1}\cdots a_{j_{d-2}}a_j,\qquad a_{j_{d+1}}\cdots a_{j_s},\qquad a_{i_1}\cdots a_{i_{s-2}}a_j,\qquad a_0.
   \]
   If $i_1\ne j_1$, we let  $k:={j_1}$, and proceed as in (ii), {taking $\alpha$ to be the $k$th projection of $a_ka_{j_2}\cdots a_{j_{d-2}}a_j$ and $\beta$ that of $( a_{i_{1}}\cdots a_{i_{s-2}}a_j)^{-1}$.}
   If $i_1= j_1$, we instead let $k:=j_{d+1}$ and likewise following (ii) we consider the $k$th level projection of $
   a_ka_{j_{d+2}}\cdots a_{j_s}$  multiplied with that of $\big( a_{i_1}\cdots a_{i_{s-2}}a_j
   \big)^{-1}$.

    \medskip
    
  We aim to continue in this manner, using newly-formed parts of $g$ to split the existing parts of ${h_0}$, and then using the newly-formed parts of ${h_0}$ to split the existing parts of $g$.
 Observe also that if $a_i$, for some $i\in [0,s-1]$, is an isolated part of $g$ (that is, a part of $g$ of length one), then using (i) or (ii), one can further split the parts of $h_0$ to isolate $a_i$ from the parts of~$h_0$. Indeed, if $a_i$ or~$a_i^{-1}$ occurs as an endpoint of a part of~$h_0$, then it is clear. If $a_i$ is an interior point of a part $a_{r_1}\cdots a_{r_\xi} a_i a_{r_{\xi +1}}\cdots a_{r_{\xi+z}}$ of~$h_0$, then projecting to the $i$th level, we have
 \[
 \big(a_{r_1 -i}\cdots a_{r_\xi -i} a_0 a_{r_{\xi +1}-i}\cdots a_{r_{\xi+z}-i}\big)a_0^{-1}\in H_{v_i},
 \]
 and thus
 \[
 a_{s-1},\qquad a_{r_1 -i-1}\cdots a_{r_\xi -i-1},\qquad a_{r_{\xi +1}-i-1}\cdots a_{r_{\xi+z}-i-1}
 \]
 are elements of $H_{v_{i+1}}$, giving
 \[
 a_{i},\qquad a_{r_1}\cdots a_{r_\xi},\qquad a_{r_{\xi +1}}\cdots a_{r_{\xi+z}}
 \]
 in $H_{v_{i+s}}$. As usual, we then replace $H$ with $H_{v_{i+s}}$. We proceed similarly in the case when $a_i^{-1}$ is an interior point in a part of~$h_0$.

 Hence we may assume that the set of length one parts of $g$ is equal to the set of length one parts of $h_0$. Equivalently, the set of parts of $g$ of length at least two involve the same generators that appear in the parts of $h_0$ of length at least two. 
  
 If there are no parts of length at least two, then all generators have been isolated, and we are done, so assume otherwise. Suppose for now that the parts of $g$ of length at least two are labelled as follows:
   \[
   a_{e_1}*\cdots *a_{f_1},\quad  a_{e_2}*\cdots *a_{f_2}, \quad \ldots \quad, \quad a_{e_\mu}*\cdots *a_{f_\mu},
   \]
   for some $1\le \mu<s$, and similarly for $h_0$:
   \[
   a_{p_1}^{\gamma_1}*\cdots *a_{q_1}^{\lambda_1},\quad  a_{p_2}^{\gamma_2}*\cdots *a_{q_2}^{\lambda_2}, \quad \ldots \quad, \quad a_{p_\nu}^{\gamma_\nu}*\cdots *a_{q_\nu}^{\lambda_\nu},
   \]
    for some $1\le \nu<s$, with $\gamma_j=1$ if $j\in[1,s-1]$ and $\gamma_j=-1$ if $j=0$ and similarly for $\lambda_j$. Here~$*$ stands for unspecified elements in the alphabet~$S$. Write
    \[
    \mathcal{E}_g=\{(a_{e_1},a_{f_1}),\ldots, (a_{e_\mu},a_{f_\mu})\}
    \]
  for the set of ordered pairs of the so-called endpoint generators. If $a_0$ has not been isolated,   it follows that the corresponding set $\mathcal{E}_{h_0}$ of endpoint generator pairs for $h_0$ is of the form
     \[
    \mathcal{E}_{h_0}=\{(a_0^{-1},a_{q_1}),(a_{p_2},a_{q_2}),\ldots, (a_{p_\nu},a_{q_\nu})\},
    \]
    subject to reordering the parts of $h_0$. Indeed,
   else we may separate the parts further using (i).
    Without loss of generality, write
     \[
     \mathcal{E}_g=\{(a_{e_1},a_0),(a_{e_2},a_{f_2}),\ldots, (a_{e_\mu},a_{f_\mu})\}.
     \]
    Note that
    if \[
    \{p_2,\ldots,p_\nu\}\cup  \{q_1,\ldots,q_\nu\} \ne  \{e_1,\ldots,e_\mu\}\cup  \{f_2,\ldots,f_\mu\},  
    \]
     we may proceed as in (i) or (ii), since then an endpoint from a part of $g$ is an interior point in a part of $h_0$, or vice versa. Hence $\mu=\nu$ and 
    \[
    \{p_2,\ldots,p_{\mu}\}\cup  \{q_1,\ldots,q_{\mu}\} = \{e_1,\ldots,e_\mu\}\cup  \{f_2,\ldots,f_\mu\}.  
    \]
    Since
  $    \{p_2,\ldots,p_\mu\}$ has less elements than    $\{e_1,\ldots,e_\mu\}$, it follows that $e_i\in \{ q_1,\ldots,q_{\mu}\}$ for some $i\in [1,\mu]$. 
  Then we proceed as in (ii). Hence, if $a_0$ is not an isolated part of~$g$ (equivalently of~$h_0$), then we can continue splitting the parts of $g$ and $h_0$.
   
    So suppose now that $a_0$ has been isolated. As reasoned above, we have
   \[
    \mathcal{E}_g=\{(a_{e_1},a_{f_1}),\ldots, (a_{e_\mu},a_{f_\mu})\}
    \]
    and 
    \[
    \mathcal{E}_{h_0}=\{(a_{p_1},a_{q_1}),\ldots, (a_{p_\mu},a_{q_\mu})\}
    \]
    with
    \[
    \{e_1,\ldots,e_\mu\}\cup  \{f_1,\ldots,f_\mu\} =  \{p_1,\ldots,p_\mu\}\cup  \{q_1,\ldots,q_\mu\}.
    \]
    Similarly if
    \[
    \{e_1,\ldots,e_\mu\}{\cap}  \{q_1,\ldots,q_\mu\}{\ne \varnothing}, 
    \]
    we proceed as in (ii). So we assume that
    \[
    \{e_1,\ldots,e_\mu\} =  \{p_1,\ldots,p_\mu\}\quad\text{and}\quad \{f_1,\ldots,f_\mu\} =  \{q_1,\ldots,q_\mu\}.
    \]

   To proceed, we now consider the element $h_{e_1}$
   defined at the beginning of the proof. Proceeding as in (i) and (ii), we use the parts of $g$ and $h_0$ to split $h_{e_1}$ into parts, and if possible, we likewise use the parts of $h_{e_1}$ to further split the parts of $g$ and $h_0$. We claim that $a_{e_1}$ has been isolated through this process. Indeed,
   analogously to the considerations above for when $a_0$ was assumed to be an endpoint in $\mathcal{E}_g$,
   if we have
    \[
    \mathcal{E}_{h_{e_1}}=\{(a_{k_1},a_{e_1}^{-1}),(a_{k_2},a_{\ell_2}),\ldots, (a_{k_\eta},a_{\ell_\eta})\}
    \]
    and
     \[
     \mathcal{E}_g=\{(a_{e_1},a_{f_1}),(a_{e_2},a_{f_2}),\ldots, (a_{e_\eta},a_{f_\eta})\},
     \]
     where here $\eta\ge\mu$, and by abuse of notation we still write $e_i$ for the left endpoints and $f_i$ for the right endpoints for the parts of~$g$.
     Then, as seen before, there is some $f_i\in\{k_1,\ldots,k_{\eta}\}$ for $i\in[1,{\eta}]$, and we can proceed as in (i) or (ii). If instead $ (a_{e_1}^{-1},a_{\ell_1})\in \mathcal{E}_{h_{e_1}}$ then we multiply the $e_1$th projection of $a_{e_1}*\cdots *a_{f_1}$ with that of $a_{e_1}^{-1}*\cdots *a_{\ell_1}$ as in (ii). Lastly, if $a_{e_1}^{-1}$ is an interior point in~$\mathcal{E}_{h_{e_1}}$, then we proceed as in (ii). In other words, if $a_{e_1}$ is not an isolated part of~$g$ (equivalently of~${h_0}$ and of~$h_{e_1}$), then we can always continue splitting.

    By abuse of notation, we redefine $\mathcal{E}_g$ to be the new set of endpoint pairs, after this further splitting of the parts of $g$. If $\mathcal{E}_g\ne\varnothing$, pick a left endpoint $a_e$ for some ${(a_e,a_*)}\in\mathcal{E}_g$. 
    From working in a similar manner with the element $h_e$, we can isolate $a_e$.

   Proceeding in this manner, we  will end up  with all individual generators.  
   
   \bigskip
   
  \underline{Case 2:}  Suppose  $d=1$. Thus we have
  \[
   g=a_{i_1}\cdots a_{i_{s-1}}a_0\quad\text{and}\quad  h_0=a_0^{-1}a_{j_{2}}\cdots a_{j_s}.
  \]
   Write $i:=i_1$ and let $r\in[2,s]$ be such that $j_r=i$. As in situation (ii), we consider instead the $i$th projection of $g$ multiplied with that of $h_0^{-1}$. We now proceed as in Case 1 with the argument using the pairs of endpoints $\mathcal{E}_g$.

   \medskip

  \underline{Case 3:}  Suppose  $d=s$.      Here we proceed first using (i), and then following the argument laid out in Case 1.
      \end{proof}

 \begin{rmk}
\label{rmk:appendix}
For the case $t=0$, as mentioned above it is entirely analogous as compared to the case $t>0$. Whenever we had applied situation (i) in the case of $t>0$, we now use situation (ii) in the case $t=0$. After the first splitting, whenever
we consider new endpoints for the next splittings, where we had considered left
endpoints in the case $t>0$, we consider right endpoints in the case $t=0$, and vice versa.
\end{rmk}

\begin{rmk}
\label{rem:normal}  
Akin to \cite[Proposition~6.6]{DNT}, one can show that the group~$\mathfrak{K(v)}$ has a non-normal maximal subgroup of index~$q$, for infinitely many primes~$q$. Indeed, the group~$\mathfrak{K(v)}$ has a proper quotient isomorphic to $W_m(\mathbb{Z})$, where for $\mathcal{G}$ a group and $m\in\mathbb{N}_{\ge 2}$, we write $W_m(\mathcal{G})$ for the wreath product
of~$\mathcal{G}$ with a cyclic group of order~$m$. Writing  $L=\psi^{-1}(\mathfrak{K(v)}'\times \cdots \times \mathfrak{K(v)}')$ and  $N=L\langle a_0^m \rangle\trianglelefteq \mathfrak{K(v)}$, analogous to \cite[Lemma~6.4]{DNT} we have that $\mathfrak{K(v)}/N\cong W_m(\mathbb{Z}^{s-1})$, which has $W_m(\mathbb{Z})$ as a quotient group; compare also \cref{Theorem}{thm:properties}(iii).
\end{rmk}


\end{document}